\documentclass[pdftex]{amsart}
\pdfoutput=1
\usepackage{amssymb,amsmath,latexsym,amscd}
\usepackage{color}
\usepackage{graphicx}
\usepackage{bbm}
\usepackage{tabularx}
\usepackage{units}
\usepackage[pdfpagelabels]{hyperref}
\usepackage{tikz,ifthen}
\usetikzlibrary{decorations.pathreplacing}
\usetikzlibrary{arrows,calc}

\tikzstyle{perspective adjusted}=[%
x={(.8,-.07)},%
y={(0,1)},%
z={(-1.3,-.2)}]

\tikzstyle{perspective nqgs}=[%
x={(.8,.3)},%
y={(.2,1.5)},%
z={(2.4,-.5)}]

\newtheorem{theorem}{Theorem}[section]
\newtheorem{corollary}[theorem]{Corollary}
\newtheorem{proposition}[theorem]{Proposition}
\newtheorem{lemma}[theorem]{Lemma}

\theoremstyle{definition}
\newtheorem{conjecture}{Conjecture}
\newtheorem{definition}[theorem]{Definition}
\newtheorem{remark}[theorem]{Remark}
\newtheorem{question}[theorem]{Question}
\newtheorem{example}[theorem]{Example}

\newcommand\up[1]{\left\lceil #1 \right\rceil}

\newcommand{\pro}[2]{\langle #1, #2 \rangle}
\newcommand{\ad}[2]{{#1}^{(#2)}}
\newcommand{\adP}[1]{P^{(#1)}}
\newcommand{\adQ}[1]{Q^{(#1)}}
\newcommand{\oni}[1]{{#1}^{[1]}}
\newcommand\Ps{\ad{P}{s}}
\renewcommand\emptyset{\varnothing}
\renewcommand\ge{\geqslant}
\renewcommand\le{\leqslant}
\renewcommand\geq{\geqslant}
\renewcommand\leq{\leqslant}

\renewcommand\tilde{\widetilde}
\newcommand\dual{\vee}
\newcommand{\NF}{\mathcal{N}}

\renewcommand\epsilon{\varepsilon}
\renewcommand\phi{\varphi}
\newcommand{\qd}{\mu}  
\newcommand{\nf}{\tau}
\newcommand\N{\mathbb N}
\newcommand\Q{\mathbb Q}
\newcommand\R{\mathbb R}
\newcommand\Z{\mathbb Z}

\renewcommand\P{\mathbb P}
\newcommand\eins{\mathbbm 1}
\newcommand\transpose{{\rm T}}

\providecommand{\intr}{\mathop{\rm int}\nolimits}%
\providecommand{\aff}{\mathop{\rm aff}\nolimits}%
\providecommand{\conv}{\mathop{\rm conv}\nolimits}%
\providecommand{\pos}{\mathop{\rm pos}\nolimits}%
\providecommand{\lcm}{\mathop{\rm lcm}\nolimits}%
\providecommand{\relint}{\mathop{\rm relint}\nolimits}%
\providecommand{\cd}{\mathop{\rm cd}\nolimits}%
\providecommand{\core}{\mathop{\rm core}\nolimits}%
\providecommand{\height}{\mathop{\rm ht}\nolimits}%

\title{Polyhedral Adjunction Theory}

\author[Di Rocco, Haase, Nill, Paffenholz]{Sandra Di Rocco, Christian Haase, Benjamin Nill, Andreas Paffenholz}

\address{Sandra Di Rocco \\ KTH Stockholm \\ Sweden}
\email{dirocco@math.kth.se}
\vspace{.1in}
\address{Christian Haase \\ Goethe-Universit\"at Frankfurt \\ Germany}
\email{haase@math.uni-frankfurt.de}
\vspace{.1in}
\address{Benjamin Nill \\ Case Western Reserve University \\
USA}
\email{benjamin.nill@case.edu}
\vspace{.1in}
\address{Andreas Paffenholz \\ TU Darmstadt \\ Germany}
\email{paffenholz@mathematik.tu-darmstadt.de}


\thanks{Di Rocco has been partially supported by VR-grants,
  NT:2006-3539 and NT:2010-5563.  Haase and Nill were supported by
  Emmy Noether fellowship HA 4383/1 and Heisenberg fellowship HA
  4383/4 of the German Research Society (DFG). Nill is supported by the US National Science Foundation (DMS 1203162). 
  Paffenholz is supported by the Priority Program 1489 of the German Research Foundation.
}

\begin{document}

\begin{abstract}
  In this paper we offer a combinatorial view on the adjunction theory
  of toric varieties. Inspired by classical adjunction theory of
  polarized algebraic varieties we explore two convex-geometric
  notions: the $\Q$-codegree and the nef value of a rational polytope
  $P$. We prove a structure theorem for lattice polytopes $P$ 
  with large $\Q$-codegree. For this, we define the adjoint polytope
  $\adP{s}$ as the set of those points in $P$ whose lattice distance
  to every facet of $P$ is at least $s$. It follows from our main
  result that if $\adP{s}$ is empty for some $s < 2/(\dim P+2)$, then the lattice polytope $P$ has lattice width one. This has consequences in Ehrhart theory
  and on polarized toric varieties with dual defect. Moreover, we
  illustrate how classification results in adjunction theory can be
  translated into new classification results for lattice polytopes.
\end{abstract}

\maketitle

\section*{Introduction}
Let $P \subseteq \R^n$ be a rational polytope of dimension $n$.
Any such polytope $P$ can be described in a unique minimal way as
\[P = \{x \in \R^n \,:\,  \pro{a_i}{x} \geq b_i,\; i=1, \ldots, m\}\]
where the $a_i$ are primitive rows of an $m \times n$ integer matrix $A$, and $b \in \Q^m$.

For any $s \geq 0$ we define the {\em adjoint polytope} $\adP{s}$ as
\[\adP{s} := \{x \in \R^n \,:\, A x \geq b+s \eins\},\]
where $\eins = (1, \ldots, 1)^\transpose$. 

We call the study of such polytopes $\adP{s}$ {\em polyhedral adjunction theory}. 
\begin{figure}[ht]
\includegraphics[height=2cm]
{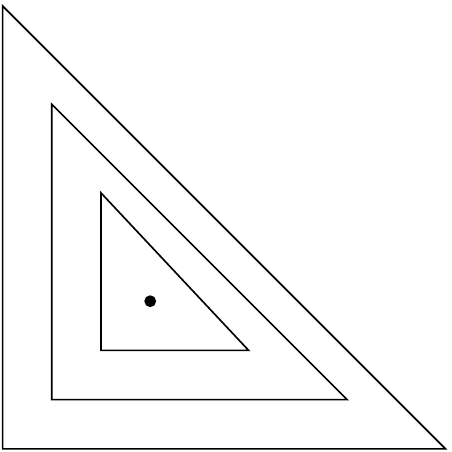}
\hspace{2cm}
\includegraphics[height=2cm]
{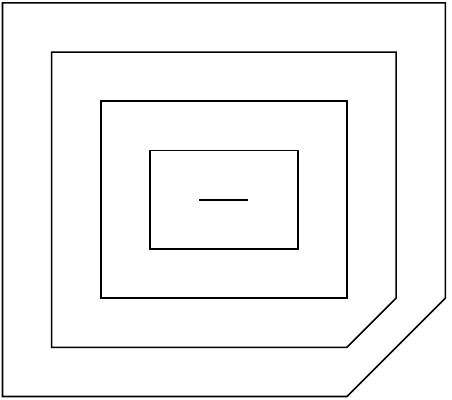}
\caption{Two examples of polyhedral adjunction}
\label{fig:start}
\end{figure}

Adjunction theory is an area of algebraic geometry which has played a
fundamental role in the classification of projective algebraic
varieties,
\cite{BT98,BDiT03,BFS92,BS94,BS95,FS89,Fuj92,Fuj96,Fuj87,Fuj97,Nak97,Som86}. The
correspondence between polarized toric varieties and lattice polytopes
provides a natural ground for an adjunction theory of lattice
polytopes, as suggested in \cite{DDP09}. 

The main purpose of this article 
is to convince the reader that polyhedral adjunction theory is an exciting area of research with many open questions connecting 
toric geometry, polyhedral combinatorics and geometry of numbers.

By the toric dictionary between convex geometry of polytopes and
geometry of projective toric varieties, a lattice polytope $P$ defines
a toric variety $X_P$ polarized by an ample line bundle $L_P$. The
pair $(X_P,L_P)$ is often referred to as a {\em polarized toric variety}.
Sometimes the pair $(X,L)$ is replaced by the equivariant embedding
$X\hookrightarrow \P^{N}$  defined by a suitable multiple of the line
bundle $L.$ 
Adjunction theory provides tools to characterize and classify the
pairs $(X,L)$ by looking at the behavior of the {\em adjoint systems}
$|uK_X+vL|,$ for integers $u,v$, where $K_X$ is the canonical divisor
in $X$. We refer to Section \ref{AG} for details.
If $P$ is the polytope defined by the line bundle $L$ on $X$, then
$\ad{(vP)}{u}$ is the polytope defined by the line bundle
$uK_X+vL.$ 

In adjunction theory the {\em nef value } $\tau(L)$ and the {\em
  unnormalized spectral value}  $\mu(L)$ (sometimes called the
canonical threshold) measure the positivity of the adjoint systems. In
Section \ref{AG} an account of these notions is given.  An `integral' 
version of the unnormalized spectral value for lattice polytopes has been present
in the literature for quite some time (even though it was never
defined this way) under the name \emph{codegree}, denoted by $\cd(P)$,
see Definition~\ref{def:codegree}.
This notion appeared in connection with Ehrhart theory and was studied by Batyrev and Nill in \cite{BN07}.

A `rational' version, again for lattice polytopes, has recently been introduced in \cite{DDP09}.
Let $c$ be the maximal rational number for which $\adP{c}$ is non-empty. 
Its reciprocal, $\qd(P):=1/c$ equals precisely the unnormalized spectral value $\mu(L_P)$. It is called the {\em $\Q$-codegree} of
$P$ (Definition~\ref{def:setting}). 

A long-standing conjecture in algebraic geometry states  that general
polarized varieties should have unnormalized spectral values that are bounded
above by approximately half their dimension.
In particular, Beltrametti and Sommese conjectured the following, see
Remark \ref{conj} for details.
\begin{conjecture}[Beltrametti \&\, Sommese~\cite{BS94}]\label{C1}
If an $n$-dimensional polarized variety $X$ is smooth, then
$\mu(L)>\frac{n+1}{2}$ implies that $X$ is a fibration.\end{conjecture}

Let us consider lattice polytopes again. A Cayley sum of $t+1$ polytopes is a polytope (denoted by  $P_0*\ldots* P_t$) 
built by assembling the polytopes $P_i$ along the vertices of a
$t$-dimensional simplex, see Definition \ref{CayleyDef}. 
For $t=0$, the condition to be a Cayley sum is vacuous. So when we say
that $P$ has a Cayley structure we mean a nontrivial one with $t>0$.
For example, for $t=1$, the condition is known in the literature as
$P$ having {\em lattice width one}.
From an (apparently) unrelated perspective Batyrev and Nill
conjectured that there is a function $f(n)$ such that, if $\cd(P)\geq
f(n)$, then the polytope has a nontrivial Cayley structure. This can
be sharpened as follows:
\begin{conjecture}[Dickenstein \&\, Nill~\cite{DN10}]\label{C2} 
  If an $n$-dimensional lattice polytope $P$ satisfies $\cd(P) >
  \frac{n+2}{2}$, then $P$ decomposes as a Cayley sum of lattice
  polytopes of dimension at most $2 (n+1-\cd(P))$.
\end{conjecture} 
The polarized toric variety associated to a Cayley polytope is
birationally fibered in projective spaces, as explained in
Section \ref{fiber}.  It follows that Conjecture \ref{C2} could be considered
an `integral-toric' version of Conjecture \ref{C1} extended to singular
varieties. It also suggests that geometrically it would make sense to replace
$\cd(P)$ by $\qd(P)$ and use the bound $(n+1)/2$ from
Conjecture~\ref{C1}. This leads to the following reformulation (we
note that $\mu(P) \le \cd(P)$):
\begin{conjecture}\label{C3} 
  If an $n$-dimensional lattice polytope $P$ satisfies $\mu(P) >
  \frac{n+1}{2}$, then $P$ decomposes as a Cayley sum of lattice
  polytopes of dimension at most $\lfloor 2 (n+1-\mu(P)) \rfloor$.
\end{conjecture}
The main result of this paper is Theorem \ref{main} which proves a
slightly weaker version of Conjecture \ref{C3}, with $\mu(P) >
\frac{n+1}{2}$ replaced by $\mu(P) \ge \frac{n+2}{2}$
(cf.~Corollary~\ref{cor}). 

Despite much work both Conjectures \ref{C1} and \ref{C2} are still
open in their original generality. It is known that $f(n)$ can be
chosen quadratic in $n$ (\cite{HNP09}) and that Conjecture $\ref{C2}$
is true for smooth polytopes (\cite{DDP09,DN10}). The results in \cite{DDP09} and \cite{DN10} also imply 
that for toric polarized manifolds Conjecture \ref{C1} holds
for $\mu(L) > \frac{n+2}{2}$.\\

Besides the underlying geometric intuition and motivation,
polyhedral adjunction theory and the results of this paper have
connections with other areas.

%\noindent{\em - geometry of numbers.} 
\subsubsection*{Geometry of Numbers}
It follows from the definition of the $\Q$-codegree that $\mu(P) > 1$
implies that $P$ is {\em lattice-free}, i.e., it has no interior
lattice points. Lattice-free polytopes are of importance in geometry
of numbers and integer linear optimization, see \cite{Averkov-etal,
  Nill-Ziegler} for recent results. Lattice-free simplices turn
up naturally in singularity theory \cite{Morrison-etal}. Most
prominently, the famous flatness theorem states that $n$-dimensional
lattice-free convex bodies have bounded lattice width (we refer to
\cite{Barvinok} for details). Cayley polytopes provide the most
special class of lattice-free polytopes: they have lattice width one,
i.e., the vertices of the polytope lie one two parallel affine
hyperplanes that do not have any lattice points lying strictly between
them. Our main result, Corollary~\ref{cor}, shows that lattice
polytopes with sufficiently large $\Q$-codegree have to be Cayley
polytopes. This hints at a close and not yet completely understood
relation between the $\Q$-codegree and the lattice width of a lattice
polytope.

Let us remark that for $n \ge 3$ Corollary~\ref{cor} only provides a
sufficient criterion for $P$ to be a Cayley polytope. For
instance, $P = [0,1]^n$ has lattice width one, but $\qd(P)=2 <
\frac{n+2}{2}$.  Still, for even $n$ the choice of $\frac{n+2}{2}$ is
tight. Let $P = 2\Delta_n$, where $\Delta_n := \conv(0, e_1, \ldots,
e_n)$ is the unimodular $n$-simplex. Here, $P$ does not have lattice
width one, since every edge contains a lattice point in the middle. On
the other hand, we have $\qd(P) = \frac{n+1}{2}$. Since for $n$ even
we have $\cd(P) = \frac{n+2}{2}$, this example also shows that the
bound $\frac{n+2}{2}$ in Conjecture~\ref{C2} is sharp.

%\noindent{\em - Projective duality.}
\subsubsection*{Projective Duality} 
There is evidence that the unnormalized spectral value is connected to
the behavior of the associated projective variety under projective duality. An
algebraic variety is said to be {\em dual defective}, if its dual
variety has codimension strictly larger than $1$. The study of dual defective projective
varieties is a classical area of algebraic geometry (starting from
Bertini) and a growing subject in combinatorics and elimination
theory, as it is related to discriminants~\cite{GKZ94}.
It is known that nonsingular dual defective polarized varieties
necessarily satisfy $\qd > \frac{n+2}{2}$ \cite{BFS92}. 
On the other hand, in \cite{DN10,DiR06} it was shown that a polarized
nonsingular toric variety corresponding to a lattice polytope $P$ as
above is dual defective if and only if $\qd > \frac{n+2}{2}$. It was
conjectured in \cite{DN10} that also in the singular toric case $\qd >
\frac{n+2}{2}$ would imply $(X_P,L_P)$ to be dual defective.
Theorem~\ref{new}
% \ref{thm:main} 
gives significant evidence in favor of this conjecture, as it was
shown in \cite{CC07,Est10} that the lattice points in such a dual
defective lattice polytope lie on two parallel hyperplanes. Moreover,
using our main result we verify a weaker version of this conjecture
(Proposition~\ref{dual}).

%\noindent{\em- Clasification of polytopes and adjunction theory
%beyond the Gorenstein class.}  
\subsubsection*{Classification of polytopes and adjunction theory
  beyond $\Q$-Gorenstein varieties}
We believe that polyhedral adjunction theory can help to develop
useful intuition for problems in (not necessarily toric) classical
adjunction theory, when no algebro-geometric tools or results exist so
far.
For instance, defining $\qd$ makes sense in the polyhedral setting
even if the canonical divisor of the toric variety is not
$\Q$-Cartier.

\subsection*{How to read this paper.} 
Sections \ref{sec:codegree}--\ref{sec:cayley}, as well as the
appendix, are kept purely combinatorial, no prior knowledge of
algebraic or toric geometry is assumed.  The algebro-geometrically
inclined reader may jump directly to Section \ref{AG}. We refer the
reader who is unfamiliar with polytopes to \cite{Zie95}.

In Section~\ref{sec:codegree} we introduce the two main players: the
$\Q$-codegree and the nef value of a rational polytope. Section
\ref{sec:natural} proves useful results about how these invariants
behave under (natural) projections.  These results should be viewed as
a toolbox for future applications.  Section \ref{sec:cayley} contains
the main theorem and its proof. The algebro-geometric background and
implications are explained in Section \ref{AG}.  In an appendix we
include a combinatorial translation of some well-known
algebro-geometric classification results by Fujita which we think may
be of interest to combinatorialists.

\subsubsection*{Acknowledgements}
  This work was carried out when several of the authors met at FU
  Berlin, KTH Stockholm and the Institut Mittag-Leffler. The authors
  would like to thank these institutions and the G\"oran Gustafsson
  foundation for hospitality and financial support.

  We thank Sam Payne for pointing out the reference \cite{Fuj87},
  Michael Burr for exhibiting the relation to the straight skeleton
  and Alicia Dickenstein for the proof of
  Proposition~\ref{def-prop}. Finally we would like to thank the
  anonymous referees for several suggestions that led to improvements
  and clarifications.

\section{The $\Q$-codegree, the codegree, and the nef value}\label{sec:codegree}

Throughout let $P \subseteq \R^n$ be an $n$-dimensional rational polytope.

\subsection{Preliminaries}

Let us recall that $P$ is a \emph{rational polytope} if the vertices
of $P$ lie in $\Q^n$. Moreover, $P$ is a \emph{lattice polytope}, if
its vertices lie in $\Z^n$. We consider lattice polytopes up to
lattice-preserving affine transformations. Let us denote by
$\pro{\cdot}{\cdot}$ the pairing between $\Z^n$ and its dual lattice
$(\Z^n)^*$.

There exists a natural lattice distance function $d_P$ on $\R^n$ such
that for $x \in \R^n$ the following holds:
$x \in P$ (respectively, $x \in \intr(P)$) 
if and only if $d_P(x) \geq 0$ (respectively, $d_P(x) > 0$).

\begin{definition}\label{inequalities}
Let $P$ be given by the inequalities
\begin{align} \label{eq:inequalities} \tag{$*$}
  \pro{a_i}{\cdot} \geq b_i
  \qquad
  \text{for } i = 1, \ldots, m
\end{align}
where $b_i \in \Q$ and the $a_i \in (\Z^n)^*$ are primitive (i.e.,
they are not the multiple of another lattice vector).
We consider the $a_i$ as the rows of an $m \times n$ integer matrix
$A$. Further, we assume all inequalities to define facets $F_i$ of $P$.
Then for $x \in \R^n$ we define the {\em lattice distance} from $F_i$
by
\[d_{F_i}(x) := \pro{a_i}{x} - b_i\]
and the {\em lattice distance} with respect to $\partial P$ by
\[d_P(x) := \min_{i = 1, \ldots, m} d_{F_i}(x).\]
For $s > 0$ we define the adjoint polytope as
\[\ad{P}{s} := \{x \in \R^n \,:\, d_P(x) \geq s\}.\]
\end{definition}
\begin{remark}
We remark that it is important to assume that all $F_i$ are {\em
  facets}, as the following two-dimensional example shows.
Let $a_1 := (-1,1)$, $a_2 := (1,2)$, $a_3 := (0,-1)$, $a_4
:= (0,1)$.
We set $b_1 := 0$, $b_2 := 0$, $b_3 := -1$, $b_4 := 0$. This defines
the lattice triangle $P := \conv((0,0), (1,1), (-2,1))$ 
having facets $F_1,F_2,F_3$, while $F_4 := \{x \in P \,:\,
\pro{a_4}{x} = 0\}$ is just the vertex $(0,0)$. Then the point $x
:= (-1/6,1/4)$ satisfies $d_P(x) = 1/3$, however $\pro{a_4}{x} =
1/4$. Note that $a_4$ is a strict convex combination of $(0,0)$, $a_1$ and $a_2$. 
It can be shown that such a behaviour cannot occur for canonical rational polytopes 
in the sense of Definition~\ref{canonical-def} below.
\end{remark}
\begin{figure}[ht]
  \includegraphics[width=5cm]{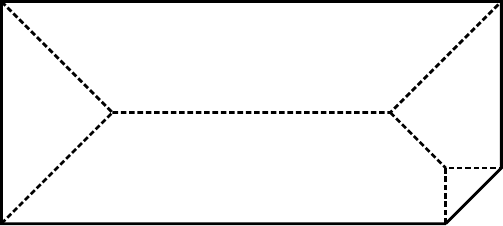}
\caption{The skeleton of vertices of the adjoint polytopes}
\label{vertex-skeleton}
\end{figure}
\begin{remark} 
As the parameter $s$ varies, the vertices of the adjoint polytopes
trace out a skeleton of straight line segments (compare
Figure~\ref{vertex-skeleton} and Lemma~\ref{lemma:QGorenstein}).  
In computational geometry there are similar constructions such as the
medial axis and the straight skeleton \cite{AA95,EE98} which are of
importance in many applications from geography to computer graphics.
``Roof constructions'' such as $M(P)$ in Proposition~\ref{prop:M(P)}
are also intensively studied in
this context (compare Figure~\ref{fig:mon}). 
The skeleton proposed here is different, since it uses a distance
function which is invariant under lattice-preserving affine transformations and not
defined in terms of Euclidean distance or angles.
\end{remark}
Let us note some elementary properties of {\em polyhedral adjunction}:
\begin{proposition}\label{prop:elem}
Let $s \geq 0$.
\begin{enumerate}
\item\label{item:adj1} Each facet of $\ad{P}{s}$ is of the form
\[\ad{F}{s} := \{x \in \ad{P}{s} \,:\, d_F(x) = s\}\]
for some facet $F$ of $P$.
\item\label{item:adj2} Assume $\ad{P}{s}$ has dimension $n$, and let $x \in
  \ad{P}{s}$. Then $d_{\ad{P}{s}}(x) = d_P(x) - s$.
  Moreover, if $x \in \intr(\adP{s})$ and $d_P(x) = d_F(x)$ for a
  facet $F$ of $P$, then $\ad{F}{s}$ is a facet of $\ad{P}{s}$, and
  $d_{\ad{P}{s}}(x) = d_{\ad{F}{s}}(x)$.
\item\label{item:adj3} Assume $\ad{P}{s}$ has dimension $n$, and let $r \geq 0$. Then
  \[\ad{(\ad{P}{s})}{r} = \ad{P}{s+r}.\]
\item\label{item:adj4} For $r > 0$ we have $r ( \ad{P}{s} ) = \ad{(r P)}{r s}$.
\end{enumerate}
\end{proposition}
\begin{proof}
 (\ref{item:adj1}) follows directly from the definition. For
  (\ref{item:adj2}), we first prove the second statement. Let $x \in
  \intr(\adP{s})$, and let $F$ be a facet of $P$ with
  $d_P(x)=d_F(x)$. If we set $\lambda := s/d_F(x)$, we have $\lambda x
  + (1-\lambda) F \subseteq \ad{F}{s}$: all elements $y$ of the left
  hand side satisfy $d_F(y)=s$ and $d_G(y) \ge s$ for facets $G$ of
  $P$ other than $F$. This shows that $\ad{F}{s}$ is indeed
  $(n-1)$-dimensional.

  This also shows that $d_P(x) = d_F(x) = d_{\ad{F}{s}}(x)+s \ge
  d_{\Ps}(x)+s$.  On the other hand, pick a facet $G$ of $P$ so that
  $\ad{G}{s}$ is a facet of $\ad{P}{s}$, and so that $d_{\ad{G}{s}}(x)
  = d_{\ad{P}{s}}(x)$. Then $d_P(x) \leq d_G(x) = d_{\ad{G}{s}}(x) + s
  = d_{\ad{P}{s}}(x) + s$.

  Finally, if $x$ sits on the boundary of $\Ps$, then the desired
  equality reads $0=0$.

  Now (\ref{item:adj3}) follows directly from (\ref{item:adj2}), and
 (\ref{item:adj4}) is immediate from the definition.
\end{proof}

\subsection{The $\Q$-codegree} 
We now define the invariant we are most interested in. The reciprocal is used to 
keep the notation consistent with already existing algebro-geometric terminology.

\begin{definition}\label{def:setting}
We define the {\em $\Q$-codegree} of $P$ as 
\[\qd(P) := (\sup\{s > 0 \,:\, \ad{P}{s} \not=\emptyset\})^{-1},\]
and the {\em core} of $P$ is $\core(P) := \adP{1/\qd(P)}$.
\end{definition}

As the following proposition shows, the supremum is actually a maximum. Moreover, 
since $P$ is a rational polytope, $\qd(P)$ is a positive rational number.

\begin{proposition}\label{prop:equi}
The following quantities coincide:
\begin{enumerate}
\item\label{item:qcd1} $\qd(P)$
\item\label{item:qcd2} $(\max\{s > 0 \,:\, \ad{P}{s} \not=\emptyset\})^{-1}$
\item\label{item:qcd3} $(\sup\{s > 0 \,:\, \dim(\ad{P}{s})=n\})^{-1}$
\item\label{item:qcd4} $\min\{p/q > 0\,:\, p,q \in \Z_{>0}, \, \ad{(p P)}{q} \not=\emptyset\}$
\item\label{item:qcd5} $\inf\{p/q > 0 \,:\, p,q \in \Z_{>0}, \, \dim(\ad{(p P)}{q})=n\}$
\item\label{item:qcd6} $\min\{p/q > 0\,:\, p,q \in \Z_{>0}, \, \ad{(p P)}{q} \cap \Z^n \not=\emptyset\}$
\end{enumerate}
Moreover, $\core(P)$ is a rational polytope of dimension $< n$.
\end{proposition}
\begin{proof}
  (\ref{item:qcd1}), (\ref{item:qcd2}), (\ref{item:qcd4}), and
 (\ref{item:qcd6}) coincide by
  Proposition~\ref{prop:elem}(\ref{item:adj4}).  For the remaining
  statements, note that for $s > 0$, the adjoint polytope $\ad{P}{s}$
  contains a full-dimensional ball if and only if there exists some
  small $\epsilon > 0$ such that $\ad{P}{s+\epsilon} \not=\emptyset$.
\end{proof}

\subsection{The codegree}

The $\Q$-codegree is a rational variant of the {\em codegree}, which
came up in Ehrhart theory of lattice polytopes \cite{BN07}. However,
the definition also makes sense for rational polytopes.
\begin{definition}\label{def:codegree}
  Let $P$ be a rational polytope. We define the {\em codegree} as
\[\cd(P) := \min\{k \in \N_{\ge 1} \,:\, \intr(kP) \cap \Z^n \not=\emptyset\}.\]
\end{definition}
\begin{lemma}
  Let $l$ be the common denominator of all right hand sides $b_i$
  given in the inequality description of $P$ as in
  \eqref{eq:inequalities} of Definition~\ref{inequalities}. Then
  \[\intr(l P) \cap \Z^n = \ad{(l P)}{1} \cap \Z^n.\]
  In particular, $\qd(P) \leq l \,\cd(P)$.
\end{lemma}
\begin{proof}
  Let $x \in\intr(l P) \cap \Z^n$. Then $\Z \ni \pro{a_i}{x} > l b_i
  \in \Z$ for all $i = 1, \ldots, m$.  Hence, $\pro{a_i}{x} \geq l b_i
  + 1$, as desired. The other inclusion is clear.  The last statement
  follows from Proposition ~\ref{prop:equi} (\ref{item:qcd6}).
\end{proof}
Note that for a \emph{lattice} polytope $P$, we automatically have
$l=1$, so
\[\qd(P) \leq \cd(P) \leq n+1,\]
where the last inequality is well-known (take the sum of $n+1$
affinely independent vertices of $P$). 

\subsection{The nef value}

The third invariant we are going to define is a finite number only if
the polytope is not too singular. Let us make this precise.

\begin{definition} \label{normal}
A rational cone $\sigma \subset (\R^n)^*$ with primitive generators
$v_1, \ldots, v_m \in (\Z^n)^*$
is \emph{$\Q$-Gorenstein of index $r_\sigma$} if there is a primitive
point $u_\sigma \in \Z^n$ with $\pro{v_i}{u_\sigma} = r_\sigma$
for all $i$.

The normal fan $\NF(P)$ of $P$ is \emph{$\Q$-Gorenstein of index $r$}
if the maximal cones 
% $\sigma_1, \ldots, \sigma_l$ 
are $\Q$-Gorenstein
%with indices $r_1, \ldots, r_l$ respectively, 
and
$r=\lcm(r_\sigma \ : \ \sigma \in \NF(P))$.

Such a cone/fan is called {\em Gorenstein}, if the index is $1$.
Moreover, we say that $P$ is {\em smooth}, if for any maximal cone of
$\NF(P)$ the primitive ray generators form a lattice basis. Clearly, $P$ smooth implies $\NF(P)$ Gorenstein.
\end{definition}
In other words, $\NF(P)$ is $\Q$-Gorenstein, if the primitive ray
generators of any maximal cone lie in an affine hyperplane, and the
index equals the least common multiple of the lattice distance of
these hyperplanes from the origin. For instance, any {\em simple}
polytope is $\Q$-Gorenstein because every cone in the normal fan is
simplicial.

\begin{definition} 
  The {\em nef value} of $P$ is given as
  \[\nf(P) := (\sup\{s > 0 \,:\, \NF(\ad{P}{s}) = \NF(P)\})^{-1} \in
  \R_{>0} \cup \{\infty\}.\]
\end{definition}
Note that in contrast to the definition of the $\Q$-codegree, here 
the supremum is never a maximum. 
\begin{definition} \label{def:QGorenstein} Assume $\NF(P)$ is
  $\Q$-Gorenstein, and $v$ is a vertex of $P$. Assume that in the
  inequality description of $P$ as in \eqref{eq:inequalities} of
  Definition~\ref{inequalities}, the vertex $v$ satisfies equality
  precisely for $i \in I$. That is, the normal cone of $v$ is $\sigma
  = \pos ( a_i \ : \ i \in I )$.  For $s \geq 0$, define the point
  $v(s)$ by $v(s) = v + \frac{s}{r_\sigma} u_{\sigma}$, where
  $u_\sigma$ and $r_\sigma$ are defined in Definition~\ref{normal}.
  Note that $\pro{a_i}{v(s)} = b_i+s$ for $i \in I$.
\end{definition}

The following lemma collects various ways to compute the nef value
$\tau$ of a polytope, if the normal fan is $\Q$-Gorenstein.
\begin{lemma} \label{lemma:QGorenstein}
$\NF(P)$ is $\Q$-Gorenstein if and only if $\nf(P) < \infty$. Assume this condition holds. 
Then, for $s \in [0,\nf(P)^{-1}]$ we have $\adP{s} = \conv( v(s) \, :
\, v \text{ vertex of } P )$.
Consequently, the following quantities coincide:
\begin{enumerate}
\item\label{item:lemma:QGorenstein1} $\nf(P)^{-1}$
\item\label{item:lemma:QGorenstein2} $\max\{s \in \Q_{>0} \,:\, v(s) \in \ad{P}{s}$ for all vertices
  $v$ of $P \}$
\item\label{item:lemma:QGorenstein3} $\min\{s \in \Q_{>0} \,:\, v(s) = v'(s)$ for two different
  vertices $v,v'$ of $P \}$
\item\label{item:lemma:QGorenstein4} $\min\{s \in \Q_{>0} \,:\, \ad{P}{s}$ is combinatorially different from
  $P \}$
\item\label{item:lemma:QGorenstein5} $\max\{s \in \Q_{>0} \,:\, \NF(P)$ refines $\NF(\ad{P}{s}) \}$
\end{enumerate}
\end{lemma}
\begin{proof}
The first assertion follows by Definition \ref{def:QGorenstein}. Notice that  $\NF(P)=\NF(\ad{P}{s})$ if and only if 
$v(s)\neq v'(s)$ for any  two different
  vertices $v,v'$ of $P.$ This implies the assertions \ref{item:lemma:QGorenstein1}$\Leftrightarrow$\ref{item:lemma:QGorenstein3}$\Leftrightarrow$\ref{item:lemma:QGorenstein4}. 
  Let now $\xi=\max\{s \in \Q_{>0} \,:\, v(s) \in \ad{P}{s}\}.$ As remarked in Definition \ref{def:QGorenstein} it is
  $\nf(P)^{-1}\leq\xi.$ On the other hand the existence  of an $s\in\Q$ such that  $\xi<s<\nf(P)^{-1}$ would lead 
  to a contradiction. In fact it would imply that $\NF(P)=\NF(\ad{P}{s})$ and the existence of a vertex $v\in P$ for which $v(s)\not\in\ad{P}{s}.$
  This proves  \ref{item:lemma:QGorenstein1}$\Leftrightarrow$\ref{item:lemma:QGorenstein2}$\Leftrightarrow$\ref{item:lemma:QGorenstein5}.
\end{proof}

Figure~\ref{fig:nqgs} shows a three-dimensional lattice polytope $P$
whose normal fan is not $\Q$-Gorenstein ($\nf(P) = \infty$).
Note that $P$ has $5$ vertices, while the adjoint polytope $\adP{c}$
(for $0 < c < \frac{1}{\qd(P)}$) has $6$ vertices.
\begin{figure}[ht]
  \centering
    \begin{tikzpicture}[perspective nqgs,scale=.27]
    % outer vertices
    \coordinate (a) at (  0,   0,   0);%
    \coordinate (b) at (  0,  10,   0);%
    \coordinate (c) at (  5,   0,   0);%
    \coordinate (d) at (  0,   0,  10);%
    \coordinate (e) at (  5,   0,  10);%

    % inner vertices
    \coordinate (f) at (  4,  1,  1);%
    \coordinate (g) at (  1,  7,  1);%
    \coordinate (h) at (  1,  1,  1);%
    \coordinate (i) at (  1,  1,  8);%
    \coordinate (j) at (  1,  7,  2);%
    \coordinate (k) at (  4,  1,  8);%

    % outer backside edges
    \draw [line width=1.2pt,color=black!60] (a) -- (c);%
    \draw [line width=1.2pt,color=black!60] (b) -- (c);%
    \draw [line width=1.2pt,color=black!60] (c) -- (e);%

    % inner backside edges
    \draw [line width=1.2pt,color=gray] (f) -- (g);%
    \draw [line width=1.2pt,color=gray] (f) -- (h);%
    \draw [line width=1.2pt,color=gray] (f) -- (k);%

    % inner faces
    \fill [fill=red!50,opacity=.3] (i) -- (j) -- (k);
    \fill [fill=red!70,opacity=.3] (g) -- (h) -- (i) -- (j);

    % inner vertices and edges
    \draw [line width=1.2pt,color=black!80] (g) -- (h);%
    \draw [line width=1.2pt,color=black!80] (g) -- (j);%
    \draw [line width=1.2pt,color=black!80] (h) -- (i);%
    \draw [line width=1.2pt,color=black!80] (i) -- (j);%
    \draw [line width=1.2pt,color=black!80] (i) -- (k);%
    \draw [line width=1.2pt,color=black!80] (j) -- (k);%
    \fill (f) [red] circle (6pt);%
    \fill (g) [red] circle (6pt);%
    \fill (h) [red] circle (6pt);%
    \fill (i) [red] circle (6pt);%
    \fill (j) [red] circle (6pt);%
    \fill (k) [red] circle (6pt);%

    % inner faces
    \fill [fill=blue!30,opacity=.3] (a) -- (b) -- (d);
    \fill [fill=blue!50,opacity=.3] (b) -- (d) -- (e);

    % outer vertices and edges
    \draw [line width=1.2pt] (a) -- (b);%
    \draw [line width=1.2pt] (a) -- (d);%
    \draw [line width=1.2pt] (b) -- (d);%
    \draw [line width=1.2pt] (b) -- (e);%
    \draw [line width=1.2pt] (d) -- (e);%
    \fill (a) [blue] circle (6pt);%
    \fill (b) [blue] circle (6pt);%
    \fill (c) [blue] circle (6pt);%
    \fill (d) [blue] circle (6pt);%
    \fill (e) [blue] circle (6pt);%
  \end{tikzpicture}
  \caption{$\adP{1/5} \subseteq P$ for a $3$-dimensional lattice polytope $P$}
  \label{fig:nqgs}
\end{figure}
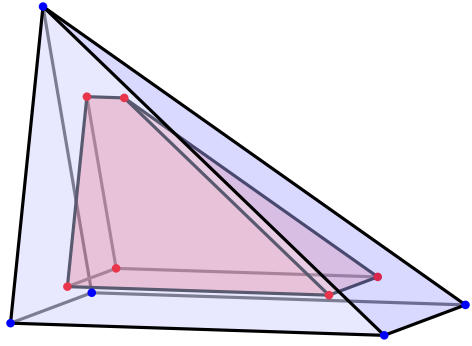

By definition, we have $\qd(P) \leq \nf(P)$. We also want to compare
the codegree and the nef value.
\begin{proposition}
Let $P$ be a lattice polytope with $\Q$-Gorenstein normal fan of index
$r$.
If $s \geq r \nf(P)$ is an integer, then $\ad{(s P)}{r}$ is a lattice
polytope. In particular, 
\[\cd(P) - 1 < r \nf(P).\]
\label{estimate}
\end{proposition}
\begin{proof}
By Lemma~\ref{lemma:QGorenstein} every vertex of $\adP{r/s}$ is
of the form $v(r/s) = v + \frac{r}{r_\sigma s} u_\sigma$ for some
vertex $v$ of $P$. Hence, every vertex of $\ad{(s P)}{r}$ is given as
$s v(r/s) = sv + \frac{r}{r_\sigma} u_\sigma$, a lattice point. For
the last statement it suffices to observe that ${(\cd(P)-1) P}$ does
not have interior lattice points.
\end{proof}

\subsection{The \emph{mountain} and $\Q$-normality}

Here is a graphical description of the nef value and the $\Q$-codegree. It also 
provides an efficient way to compute these invariants.

\begin{proposition}\label{prop:M(P)}
Let the \emph{mountain} $M(P) \subseteq \R^{n+1}$ be defined as
\[M(P) := \{(x,s) \,:\, x \in P,\; 0\leq s \leq d_P(x)\}.\]
Assume that $P$ has an inequality description as in 
\eqref{eq:inequalities} of Definition~\ref{inequalities}, then
\[M(P) = \{(x,s) \in \R^{n+1} \,:\, (A \,\mid - \eins)\,
(x,s)^\transpose \geq b, \; s \geq 0\}.\]
Therefore, $M(P)$ is a rational polytope with $M(P) \,\cap\,\ \R^n
\times \{s_0\} = \adP{s_0} \times \{s_0\}$. Moreover, 
\begin{align} \label{eqqq1}
  \qd(P)^{-1} =  \max(s \,:\, \textit{ there is a vertex of } M(P) \text{ with last coordinate } s)
\end{align}
If $\NF(P)$ is $\Q$-Gorenstein, then
\begin{equation}\label{eqqq2}
\nf(P)^{-1} =  \min(s > 0\,:\, \textit{ there is a vertex of } M(P) \text{ with last coordinate } s)
\end{equation}
\end{proposition}
\begin{proof}
  Abbreviate $q:=\qd(P)^{-1}$.  According to
  Proposition~\ref{prop:equi}(\ref{item:qcd2}), $q = \max\{s > 0 : \adP{s} \neq
  \emptyset \}$. By definition of $\adP{s}$, this is the maximal
  positive $s$ such that there is an $x \in P$ which satisfies $d_F(x)
  \ge s$ for all facets $F$ of $P$. This shows (\ref{eqqq1}).

Let us prove (\ref{eqqq2}). Suppose $\NF(P)$ is $\Q$-Gorenstein, and abbreviate $t:=\nf(P)^{-1}$.
For every vertex $v$ of $P$ and $s > 0$ define $v(s)$ as in
Definition~\ref{def:QGorenstein}.
At every vertex $(v,0)$ of the bottom facet $P \times \{0\}$ of
$M(P)$ there is a unique upwards edge towards $(v(s),s)$ for small
$s$. 
By Lemma~\ref{lemma:QGorenstein}(\ref{item:lemma:QGorenstein3})
there are two vertices $v$, $v'$ of $P$ so that $v(t) =
v'(t)$. The corresponding point $(v(t),t) = (v'(t),t)$ in $M(P)$ is a
vertex as it is incident to at least two edges.
\end{proof}

%%
%% ///(1)
%%
Let us consider the example given on the right hand side of
Figure~\ref{fig:start}, and take a look at its mountain, see Figure~\ref{fig:mon}. 
The height of the mountain equals the reciprocal of the $\Q$-codegree,
while the height of the first nontrivial vertex is the reciprocal of
the nef value.

\begin{figure}[ht]
  \includegraphics[height=2cm]{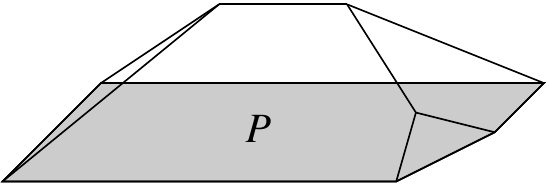}
\caption{The ``lattice distance mountain'' $M(P)$}
\label{fig:mon}
\end{figure}

% As one sees, from the second highest vertex to the top the mountain
% looks benign. 
This motivates the following definition (see \cite{DDP09}).

\begin{definition}
We say that $P$ is {\em $\Q$-normal}, if 
$\qd(P) = \nf(P)$.
\end{definition}
To get the correct intuition for this notion let us note that $P$ is $\Q$-normal if and only if all {\em vertices} of $P$ survive under polyhedral adjunction (as long as the adjoint polytope is full-dimensional). For $n \geq 3$ it is not enough that all {\em facets} of $P$ survive, as Figure~\ref{fig:pyramid} illustrates (where $\nf(P)^{-1} = 2$, $\qd(P)^{-1} = 6$ and $\core(P)$ is an interval). 
% However, this condition is equivalent to $\Q$-normality in the
% special case that $\core(P)$ is a point.
\begin{figure}[ht]
  \centering
    \begin{tikzpicture}[perspective adjusted,scale=.18]
      
    % outer vertices
    \coordinate (a) at (  0,   1,   1);%
    \coordinate (b) at (  0,   1,  -1);%
    \coordinate (c) at ( 22, -10,  12);%
    \coordinate (d) at ( 22, -10, -12);%
    \coordinate (e) at (-22, -10, -12);%
    \coordinate (f) at (-22, -10,  12);%

    % inner vertices
    \coordinate (g) at ( -2, -2,  0);%
    \coordinate (h) at (  2, -2,  0);%
    \coordinate (i) at ( 10, -6, -4);%
    \coordinate (j) at ( 10, -6,  4);%
    \coordinate (k) at (-10, -6, -4);%
    \coordinate (l) at (-10, -6,  4);%

    % backside edges of outer
    \draw [line width=1.5pt,color=gray] (d) -- (e);%
    \draw [line width=1.5pt,color=gray] (b) -- (e);%
    \draw [line width=1.5pt,color=gray] (f) -- (e);%

    % backside edges of inner
    \draw [thick,color=black!70] (g) -- (k);%
    \draw [thick,color=black!70] (i) -- (k);%
    \draw [thick,color=black!70] (k) -- (l);%

    % fill two faces of the inner
    \fill [fill=red!50,opacity=.7] (g) -- (h) -- (j) -- (l);
    \fill [fill=red!30,opacity=.7] (h) -- (i) -- (j);

    % vertices and edges of inner
    \draw [thick] (g) -- (h);%
    \draw [thick] (h) -- (i);%
    \draw [thick] (h) -- (j);%
    \draw [thick] (i) -- (j);%
    \draw [thick] (g) -- (l);%
    \draw [thick] (j) -- (l);%
    \fill (g) [red] circle (10pt);%
    \fill (h) [red] circle (10pt);%
    \fill (i) [red] circle (10pt);%
    \fill (j) [red] circle (10pt);%
    \fill (k) [red] circle (10pt);%
    \fill (l) [red] circle (10pt);%

    % outer faces
    \fill [fill=blue!50,opacity=.3] (a) -- (c) -- (f);
    \fill [fill=blue!30,opacity=.3] (a) -- (b) -- (d) -- (c);

    % outer vertices and edges
    \draw [line width=1.5pt] (a) -- (b);%
    \draw [line width=1.5pt] (a) -- (c);%
    \draw [line width=1.5pt] (b) -- (d);%
    \draw [line width=1.5pt] (a) -- (f);%
    \draw [line width=1.5pt] (c) -- (f);%
    \draw [line width=1.5pt] (d) -- (c);%
    \fill (a) [blue] circle (12pt);%
    \fill (b) [blue] circle (12pt);%
    \fill (c) [blue] circle (12pt);%
    \fill (d) [blue] circle (12pt);%
    \fill (e) [blue] circle (12pt);%
    \fill (f) [blue] circle (12pt);%
  \end{tikzpicture}
  \caption{$\adP 4 \subseteq P$ for a $3$-dimensional lattice polytope $P$}
  \label{fig:pyramid}
\end{figure}
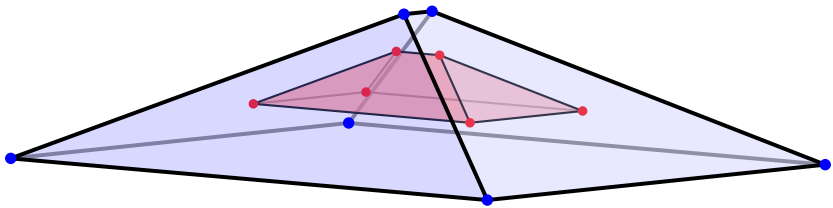
%\begin{figure}[ht]
%  \includegraphics[height=6cm]{pyramid9}
%\caption{$\adP{4} \subset P$ for a $3$-dimensional lattice polytope %$P$}
%\label{pyramid}
%\end{figure}

\section{Natural projections}\label{sec:natural}

Throughout let $P \subseteq \R^n$ be an $n$-dimensional rational polytope. 

\subsection{The core and the natural projection}

Recall that $\core(P) := \adP{1/\qd(P)}$ is a rational polytope of dimension $<n$.
\begin{definition}
Let $K(P)$ be the linear space parallel to $\aff(\core(P))$.
We call $\pi_P \colon \R^n \to \R^n/K(P)$ the \emph{natural projection\/}
associated with $P$.
\end{definition}

\begin{lemma}
Let $x \in \relint(\core(P))$. Let us denote by $F_1, \ldots, F_t$ the facets of $P$ with $d_{F_i}(x)
= \qd(P)^{-1}$. Then their primitive inner normals $a_1, \ldots, a_t$
positively span the linear subspace $K(P)^\perp$.

Moreover, if $\core(P) = \{x\}$, then
\[\{y \in \R^n \,:\, d_{F_i}(y) \geq 0 \text{ for all } i = 1, \ldots, t\}\]
is a rational polytope containing $P$.
\label{lemma:span}
\end{lemma}
\begin{proof}
We set $s := \qd(P)^{-1}$. Let $i \in \{1, \ldots, t\}$. Since
$d_{F_i}(x) = s$ and $x \in \relint(\ad{P}{s})$, we have $d_{F_i}(y) =
s$ for all $y \in \ad{P}{s}$.
This shows $C := \pos(a_1, \ldots, a_t) \subseteq K(P)^\perp$. 
Assume that this inclusion were strict. Then there exists some $v \in
\R^n$ such that $\pro{v}{C} \geq 0$ and $v$ does not vanish on the
linear subspace $K(P)^\perp$.
In particular, for any $i \in \{1, \ldots, t\}$ one gets $\pro{v}{a_i}
\geq 0$, so $d_{F_i}(x + \epsilon v) \geq d_{F_i}(x) = s$ for any
$\epsilon > 0$. Moreover, if we choose $\epsilon$ small enough, then
$d_{G}(x + \epsilon v) \approx d_G(x) > s$ for any other facet $G$ of
$P$. Hence, $x + \epsilon v \in \Ps$. But this means $v \in K(P)$, and
$v$ must vanish on $K(P)^\perp$, a contradiction.

Finally, notice that, if $\Ps = \{x\}$, then $a_1, \ldots, a_t$
positively span $(\R^n)^*$.
In particular, $\conv(a_1, \ldots, a_t)$ contains a small
full-dimensional ball around the origin.
Dually, $\{y \in \R^n \,:\, \pro{a_i}{y} \geq b_i, \, i = 1, \ldots,
t\}$ is contained in a large ball. Hence, it is a bounded rational
polyhedron, thus a rational polytope.
\end{proof}

\subsection{The $\Q$-codegree under natural projections}

We begin with a  key observation.
\begin{proposition}\label{prop:projection}
The image $Q := \pi_P(P)$ of the natural projection of $P$
is a rational polytope satisfying $\qd(Q) \geq \qd(P)$. 
Moreover, if $\qd(Q) = \qd(P)$, then $\core(Q)$ is the point
$\pi_P(\core(P))$.
\end{proposition}
\begin{proof}
Let $t,x,F_i,a_i$ as in Lemma~\ref{lemma:span} and $s := \qd(P)^{-1}$. 
$Q$ is a rational polytope with respect to the lattice $L := \Z^n/(K(P) \cap \Z^n)$.  
The dual lattice of $L$ is $(\Z^n)^* \cap K(P)^\perp$. In particular, 
any $a_i$ for $i \in \{1, \ldots, t\}$ is still a primitive normal vector of a facet of $Q$. 
In particular, $\ad{Q}{s} \subseteq \pi_P(\Ps) =
\{\pi_P(x)\}$. Therefore, $\qd(Q)^{-1} \leq s$.
\end{proof}

The following example shows that this projection can be quite
peculiar.

\begin{figure}[ht]
\includegraphics{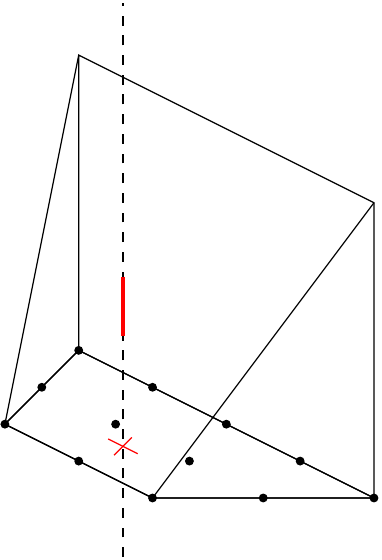}
\hspace{2cm}
\includegraphics{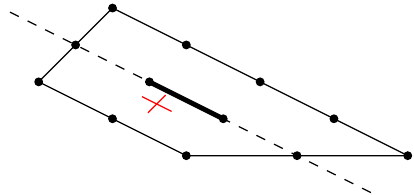}
\\
$P = \conv \left[
  \begin{smallmatrix}
    0&2&0&2&0&0 \\
    0&0&4&2&0&4 \\
    0&0&0&0&h&h
  \end{smallmatrix}
\right]$
\hspace{2cm}
$Q = \conv \left[
  \begin{smallmatrix}
    0&2&0&2 \\
    0&0&4&2
  \end{smallmatrix}
\right]$
\caption{\label{figEx}The $\Q$-codegree projection $\pi_P\;:\; P \to Q$}
\label{strange}
\end{figure}

In this picture, the dashed lines are the affine hulls along which 
we are projecting, while the fat line segments are the cores of $P$
and $Q$. 
On the left side we only drew the lattice points on the bottom face for clarity. 
Here, $\pi_P$ projects onto the bottom face $Q$. If we assume that the
height $h$ of $P$ is large enough, then the adjoint polytope
$\core(P)$ is a line segment projecting onto the point $x =
(4/3,4/3,0)$ marked on the bottom. Note that this point doesn't even
lie in the line segment $\core(Q)$. Essentially, the reason for this
peculiar behaviour is that the preimage of one of the two facets of
$Q$ defining the affine hull of $\core(Q)$ is not a facet of $P$.
Moreover, $\qd(Q) = 1 > \frac{3}{4} = \qd(P)$.

\subsection{Projections of $\alpha$-canonical polytopes}

\begin{definition}
  Let $\sigma$ be a rational cone with primitive generators $v_1,
  \ldots, v_m$. Then the height function associated with $\sigma$ is
  the piecewise linear function
  $$\height_\sigma(x) := \max \left\{ \sum_{i=1}^m \lambda_i \ : \ \lambda_i
  \ge 0 \text{ for } i = 1, \ldots, m\ , \ \sum_{i=1}^m \lambda_i v_i
  = x \right\}$$
  on $\sigma$.
  For $\alpha > 0$, we say that $\sigma$ is {\em $\alpha$-canonical} if
  $\height_\sigma(x) \ge \alpha$ for every non-zero $x \in \sigma \cap
  \Z^n$. A $1$-canonical cone is said to be {\em canonical}.

  A rational polytope is {\em ($\alpha$-)canonical} if all cones of its
  normal fan are.
\label{canonical-def}
\end{definition}
This is a generalization to the non-$\Q$-Gorenstein case of canonical
singularities in algebraic geometry.
Note that a $\Q$-Gorenstein cone of index $r$ is $1/r$-canonical.
In particular, rational polytopes with Gorenstein normal fan are canonical.

\begin{lemma} \label{lemma:canonical+projection}
  Let $\pi \colon P \to Q$ be a polytope projection, and assume $P$ is
  $\alpha$-canonical. Then $\alpha\, d_P(x) \le d_Q(\pi(x))$ for all $x
  \in P$.
\end{lemma}

\begin{proof} Let $\pro{a}{\cdot} \ge b$ be a facet of $Q$ realizing
  $d_Q(\pi(x))$. That is, $\pro{a}{\pi(x)} = b + d_Q(\pi(x))$.
  Then the integral linear functional $\pi^*a$ belongs to some cone
  $\sigma \in \NF(P)$ with primitive generators $a_1, \ldots,
  a_m$. Write $\pi^*a = \sum_{i=1}^m \lambda_i a_i$ with $\lambda_i
  \ge 0$ for $i = 1, \ldots, m$ and $\sum_{i=1}^m \lambda_i =
  \height_\sigma(\pi^*a)$. Then $b = \sum_{i=1}^m \lambda_i b_i$, and $\sum_{i=1}^m \lambda_i
  \ge \alpha$. Thus
  \begin{multline*}
    d_Q(\pi(x)) 
    = \pro{a}{\pi(x)} - b 
    = \pro{\pi^*a}{x} - b 
    \\
    = \sum_{i=1}^m \lambda_i (\pro{a_i}{x} - b_i)
    \ge \sum_{i=1}^m \lambda_i d_P(x)
    \ge \alpha d_P(x) \ .
  \end{multline*}
\end{proof}

\begin{corollary}
  Let $\pi \colon P \to Q$ be a polytope projection, and assume $P$ is
  $\alpha$-canonical. Then $\qd(P) \ge \alpha \qd(Q)$.

  In particular, if $P$ is canonical, then $\qd(P) \ge \qd(Q)$.
\end{corollary}

This shows that for {\em canonical} polytopes the natural projection in
Proposition~\ref{prop:projection} is $\Q$-codegree preserving! In 
particular, the polytope $Q$ has the nice property that $\core(Q)$ is a point. 

\begin{example}
  Unfortunately, it is in general not true that
  being $\alpha$-canonical is preserved under the natural
  projection, as can be seen from the following example. Consider
  the polytope
    \begin{align*}
      P = \conv \left[
          \renewcommand{\arraycolsep}{1.5pt}
          \renewcommand{\arraystretch}{.7}
        \begin{array}{>{\scriptstyle}r>{\scriptstyle}r>{\scriptstyle}r>{\scriptstyle}r>{\scriptstyle}r>{\scriptstyle}r>{\scriptstyle}r>{\scriptstyle}r>{\scriptstyle}r}
          14 & 8 & 0 & -8 & 14 & 0 & 0 & -14 & -14\\
          7 & 1 & 0 & 1 & 7 & 21 & 21 & 7 & 7\\
          -21 & -3 & 0 & 3 & 21 & 21 & -21 & 21 & -21
        \end{array}
      \right]
    \end{align*}
  This is a three-dimensional lattice polytope.  Its core face has the
  vertices $(0,7,7)$ and $(0,7,-7)$, so the natural projection $\pi$
  maps onto a two-dimensional lattice polytope by projecting onto the
  first two coordinates. The projection is
  \begin{align*}
    \pi(P) = \conv \left[
          \renewcommand{\arraycolsep}{1.5pt}
          \renewcommand{\arraystretch}{.7}
        \begin{array}{>{\scriptstyle}r>{\scriptstyle}r>{\scriptstyle}r>{\scriptstyle}r>{\scriptstyle}r>{\scriptstyle}r}
        14 & 8 & 0 & -8 & 0 & -14\\
        7 & 1 & 0 & 1 & 21 & 7
      \end{array}
    \right]
  \end{align*}
All but one normal cone of $P$ is canonical. The exception is the
normal cone at the origin. Its primitive rays are $(-1,-5,-1)$,
$(1,-5,1)$, $(0,-3,-1)$, and $(0,-3,1)$. The ray $(0,-1,0)$ is in the
cone, and its height is $\frac13$. So $P$ is $\frac13$-canonical.  The
normal cones of the natural projection $Q$ are again canonical with
one exception. The normal cone at the origin is generated by the rays
$(1,-8)$ and $(-1,-8)$. It contains the ray $(0,-1)$, so $Q$ is only
$\frac18$-canonical. The computations were done with
\texttt{polymake}~\cite{Joswig2009}.
\end{example}

\subsection{$\Q$-normality under natural projections}

\begin{proposition} \label{nice}
  Let $P$ be $\Q$-normal. Then its image $Q$ under the natural
  projection is $\Q$-normal, its core is the point $\core(Q) =
  \pi_P(\core(P))$, and $\qd(Q) = \qd(P)$. Moreover, if $P$ is
  $\alpha$-canonical, then $Q$ is $\alpha$-canonical.
\end{proposition}

\begin{proof}
If $P$ is $\Q$-normal, then the normal fan of $P$ refines the normal
fan of $\core(P) = \adP{1/\nf(P)}$.
In particular, the face $K(P)^\perp$ of $\NF(\core(P))$
is a union of faces of $\NF(P)$. Therefore, being $\alpha$-canonical
is preserved.
On the other hand, $\NF(Q) = \NF(P) \cap
K(P)^\perp$ for any polytope projection $P \to Q$. That means that
every facet $F$ of $Q$ lifts to a facet $\pi_P^*F$ of $P$. Together with
$d_F(\pi_P(x)) = d_{\pi_P^*F}(x)$ (for $x \in P$) this implies
$\adQ{s} = \pi(\adP{s})$ for any $s \geq 0$.
This yields the statements.
\end{proof}

If a rational polytope is $\Q$-normal and its core is a point, then 
the generators of its normal fan form the vertex set of a lattice polytope. 
Such a fan corresponds to a so-called {\em toric Fano variety}, see, e.g., \cite{Deb03,Nil05}.

\section{Cayley decompositions}\label{sec:cayley}

Throughout let $P \subseteq \R^n$ be an $n$-dimensional lattice polytope. 

\subsection{Lattice width, Cayley polytopes and codegree}

We recall that the {\em lattice width} of a polytope $P$ is defined as
the minimum of $\max_{x \in P} \pro{u}{x} - \min_{x\in P} \pro{u}{x}$
over all non-zero integer linear forms $u$.
We are interested in lattice polytopes of lattice width one, which we
also call (nontrivial) {\em Cayley polytopes} or Cayley polytopes of
length $\geq 2$.

\begin{definition}\label{CayleyDef}
Given lattice polytopes $P_0, \ldots, P_t$ in $\R^k$, 
the {\em Cayley sum} $P_0 * \cdots * P_t$ is defined to be the convex hull of 
$(P_0 \times 0) \cup (P_1 \times e_1) \cdots \cup (P_t \times e_t)$ in
$\R^k \times \R^t$ for the standard basis $e_1, \ldots, e_t$ of $\R^t$. 

We say that $P \subseteq \R^n$ is a {\em Cayley polytope of length $t+1$}, if 
there exists an affine lattice basis of $\Z^n \cong \Z^k \times \Z^t$ identifying $P$ with the Cayley sum $P_0 *  \cdots *
P_t$ for some lattice polytopes $P_0, \ldots, P_t$ in $\R^k$.
\end{definition}

This definition can be reformulated, \cite[Proposition~2.3]{BN08}.

\begin{lemma} \label{lemma:Batyrev-Nill}
  Let $\sigma \subseteq \R^{n+1}$ be the cone spanned by $P \times
  1$. 
%   Let us define $u \in (\Z^{n+1})^*$ such that $\pro{u}{P \times 1} =
%   1$.
  Then the following statements are equivalent:
  \begin{enumerate}
  \item\label{item:lemma:Batyrev-Nill1} $P$ is a Cayley polytope $P_0
    * \cdots * P_t$ of length $t+1$ 
  \item\label{item:lemma:Batyrev-Nill2} There is a lattice projection
    $P$ onto a unimodular $t$-simplex 
  \item\label{item:lemma:Batyrev-Nill3} There are nonzero $x_1,
    \ldots, x_{t+1} \in \sigma^\vee \cap (\Z^{n+1})^*$ such 
    that \[x_1 + \cdots + x_{t+1} = e_{n+1}\]
  \end{enumerate}
\end{lemma}
Since the $t$-th multiple of a unimodular $t$-simplex contains
no interior lattice points, we conclude from Lemma \ref{lemma:Batyrev-Nill}(2) that
\[\cd(P_0 * \cdots * P_t) \geq t+1.\]
Conversely, Conjecture~\ref{C2} states that having large codegree implies being a Cayley polytope. To get the reader acquainted with Conjecture~\ref{C2}, we include a simple observation. 
\begin{lemma}
If $\cd(P)>\up{\frac{n+1}{2}}$, then through every vertex there is an edge whose only lattice points are its two vertices.
\end{lemma}

\begin{proof}
Assume otherwise. Then there exists an injective lattice homomorphism $f$ mapping 
$2 \Delta_n \to P$. Therefore, Stanley's monotonicity theorem \cite{Sta93,BN07} yields $n+1 - \cd(f(2 \Delta_n)) \leq n+1 - \cd(P)$, hence 
$\cd(P) \leq \cd(f(2 \Delta_n)) \leq \cd (2 \Delta_n) = \up{\frac{n+1}{2}}$. 
This yields a contradiction to our assumption.
\end{proof}

\subsection{The decomposition theorem}

Let $P,P'$ be $n$-dimensional lattice polytopes. We will say that $P$ and $P'$ are {\em unimodularly equivalent} ($P \cong P'$), 
if there exists an affine lattice automorphism of $\Z^n$ mapping the
vertices of $P$ onto the vertices of $P'$. 
It is a well-known result, see e.g. \cite{BN07}, that $P \cong \Delta_n$ if and only if $\cd(P)=n+1$. 
Since $\mu(P) \le \cd(P) \le n+1$, and $\mu(\Delta_n)=n+1$, we deduce that $P \cong \Delta_n$ if and only if $\mu(P)=n+1$. 

The following proves a general structure result on lattice
polytopes of high $\Q$-codegree.
We set
\[d(P) := \left\{\begin{array}{rl}
2(n-\lfloor \qd(P) \rfloor) & \text{, if } \qd(P) \not\in \N\\
2 (n-\qd(P)) + 1 & \text{, if } \qd(P) \in \N
\end{array}\right.\]
If we exclude the special situation $P \cong \Delta_n$, we have
%$k \leq 2 (n+1 - \lceil \qd(P) \rceil)
$1 \leq d(P) < 2 (n+1-\qd(P))$. 

\begin{theorem}\label{main}
Let $P$ be an $n$-dimensional lattice polytope with $P \not\cong
\Delta_n$.

If $n >d(P)$, then $P$ is a Cayley sum of lattice polytopes in $\R^m$
with $m \leq d(P)$.
\label{new}
\end{theorem}
For the proof we recall the following folklore result.
\begin{lemma}\label{sums}
  Let $P \subseteq \R^n$ be an $n$-dimensional lattice polytope.  Let
  $z \in \pos(P \times \{1\}) \cap \Z^{n+1}$. Then there exist (not
  necessarily different) vertices $v_1, \ldots, v_g$ of $P$ and a
  lattice point $p \in (j P) \cap \Z^n$ with
  \[z = (v_1, 1) + \cdots + (v_g,1) + (p,j)\] such that $(p,j)=(0,0)$
  or $1 \leq j \leq n+1-\cd(P)$.
%and $p$ is not a non-negative integer  combination of $x_1, \ldots, x_g$.
\end{lemma}
\begin{proof}
  There exists an $m$-dimensional simplex $S$ in $P$ with vertices
  $v_1, \ldots, v_{m+1}$ in the vertex set of $P$ such that $z \in
  \pos((v_1,1), \ldots, (v_{m+1},1))$.  We can write 
  \begin{align*}
    z\ &=\ \sum_{i=1}^{m+1} k_i (v_i,1)\ +\ \sum_{i=1}^{m+1} \lambda_i (v_i,1)
    \quad\text{for}\quad k_i\; \in\; \N \quad\text{and}\quad \lambda \;\in\; [0,1)\,.
  \end{align*}
  See also Figure~\ref{fig:zdecomp}. 
The lattice point $\sum_{i=1}^{m+1} \lambda_i (v_i,1)$ is an element
of the fundamental parallelepiped of the simplex $S$. By
\cite[Corollary 3.11]{BR07} its height $j$ equals at most the degree
of the so-called Ehrhart
$h^*$-polynomial. Ehrhart-Macdonald-Reciprocity implies that this
degree is given by $m+1-\cd(S)$. We refer to \cite{BN07} for more
details. Now, the result follows from $j \le m+1-\cd(S) \leq
n+1-\cd(P)$ by Stanley's monotonicity theorem 
  \cite{Sta93}.
\end{proof}
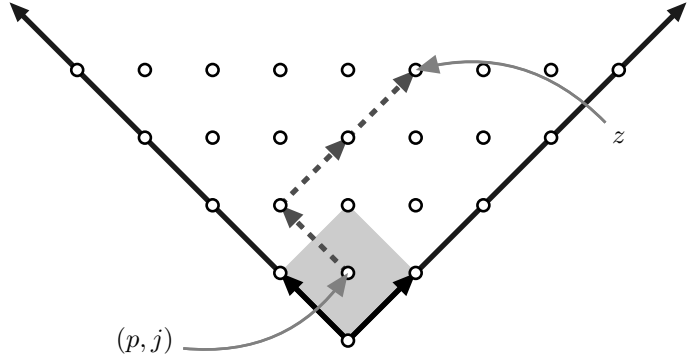
\begin{figure}[tb]
  \centering
  \begin{tikzpicture}[scale=.9]
    \foreach \x in {-1,...,9} {
      \foreach \y in {0,...,8} {
        \coordinate (a\x\y) at (\x,\y);
      }
    }

    \draw[line width=0pt,color=black!20,fill=black!20] (a40) -- (a31) -- (a42) -- (a51) -- cycle;

    \draw[line width=.7pt,black!90,arrows=-triangle 45, postaction={draw, line width=2pt, shorten >=3pt, -}] (a40) -- (a-15);
    \draw[line width=.7pt,black!90,arrows=-triangle 45, postaction={draw, line width=2pt, shorten >=3pt, -}] (a40) -- (a95);

    \draw[line width=1pt,black!70,dashed,arrows=-triangle 45, postaction={draw, line width=2pt, shorten >=3pt, -}] (a41) -- (a32);
    \draw[line width=1pt,black!70,dashed,arrows=-triangle 45, postaction={draw, line width=2pt, shorten >=3pt, -}] (a32) -- (a43);
    \draw[line width=1pt,black!70,dashed,arrows=-triangle 45, postaction={draw, line width=2pt, shorten >=3pt, -}] (a43) -- (a54);

    \draw[line width=1pt,arrows=-triangle 45, postaction={draw, line width=2pt, shorten >=3pt, -}] (4,0) -- (3,1);
    \draw[line width=1pt,arrows=-triangle 45, postaction={draw, line width=2pt, shorten >=3pt, -}] (4,0) -- (5,1);

    \foreach \x/\y in {4/0,3/1,4/1,5/1,2/2,3/2,4/2,5/2,6/2,1/3,2/3,3/3,4/3,5/3,6/3,7/3,0/4,1/4,2/4,3/4,4/4,5/4,6/4,7/4,8/4} {
      \draw[line width=1pt,fill=white]  (a\x\y) circle (2.5pt);% 
    }

    \node at (a10) {$(p,j)$} edge[line width=1pt,black!50,bend right=30,arrows=-triangle 45, postaction={draw, line width=1pt,  shorten >=3pt, -}] (a41);
    \node at (a83) {$z$} edge[line width=1pt,black!50,bend right=30,arrows=-triangle 45, postaction={draw, line width=1pt, shorten >=3pt, -}] (a54);
  \end{tikzpicture}

  \bigskip

  \caption{Decomposing $z$ in the proof of Lemma~\ref{sums}.}
  \label{fig:zdecomp}
\end{figure}

\begin{proof}[Proof of Theorem~\ref{main}]
By successive application of Proposition~\ref{prop:projection} we can
find a lattice projection $P \to Q$ with $\dim(Q)=n' \leq n$ such that
$\qd(P) \leq \qd(Q)$ and $\ad{Q}{s} = \{x\}$ for $s :=
\qd(Q)^{-1}$. By observing that $d(Q)+(n-n') \leq d(P)$,
we see that $d(P)<n$ implies $d(Q)<n'$ and, moreover, if the desired
statement holds for $Q$, then it also holds for $P$.
Hence, we may assume that $s = \qd(P)^{-1}$ and $\ad{P}{s} = \{x\}$.

By Lemma~\ref{lemma:span}, $P$ is contained in a rational polytope
$\tilde P$ with $s = \qd(\tilde P)^{-1}$ and $\ad{\tilde P}{s} = \{x\}$
so that all facets of $\tilde P$ have distance $s$ from $x$.
Let $\sigma \subseteq \tilde\sigma \subseteq \R^{n+1}$ be the
(full-dimensional, pointed) cones over $P \times \{1\} \subseteq
\tilde P \times \{1\}$, and let $u \in (\R^{n+1})^*$ be the last
coordinate functional. As $u$ evaluates positively on all vertices of
$\tilde P \times \{1\}$, we have $u \in \intr\tilde\sigma^\dual
\subseteq \intr\sigma^\dual$.
Let us define the lattice polytope
\[R := \conv(\{0\} \cup \{\eta \ : \ \eta \text{ primitive facet
  normal of } \tilde\sigma \}) \subseteq (\R^{n+1})^*.\]
In order to invoke Lemma~\ref{lemma:Batyrev-Nill}(3), we will show
that $R$ has high codegree so that $u$ can be decomposed into a sum of
many lattice points in $\tilde\sigma^\dual \subseteq \sigma^\dual$ by Lemma~\ref{sums}.

To this end, observe that $\pro{\eta}{(x,1)} = s$ for every primitive
facet normal $\eta$ of $\tilde\sigma$, so that $R$ is an
($n+1$)-dimensional pyramid with apex $0$:
$$
R = \tilde\sigma^\dual \ \cap \ \{ y \in (\R^{n+1})^* \ : \
\pro{y}{(x,1)} \le s \}\,.$$

\begin{center}
\begin{tikzpicture}[x=.3cm,y=-.3cm]

% objects at depth 50:
\draw[black] (8,16) -- (2,4);
\draw[black] (8,16) -- (14,4);
\draw[dashed,black] (8,16) -- (2,6);
\draw[line width=4.5bp,black] (7,14) -- (9,14);
\draw[black] (6.8,14) -- (9.5,14);
\draw[dashed,black] (8,16) -- (14,8);
\draw[very thick,arrows=-to,red] (13.4,14) -- (13.4,12);
\draw[very thick,red] (13,14) -- (13.8,14);
\draw[very thick,arrows=-to,red] (11.2,11.6) -- (9.6,10.4);
\draw[very thick,red] (10.9,12) -- (11.5,11.2);
\path (6,14.4) node[text=black,anchor=base east] {\fontsize{10.8}{12.96}\selectfont{}$\tilde P \times 1 \supseteq P \times 1$};
\path (4.5,7.3) node[text=black,anchor=base west] {\fontsize{10.8}{12.96}\selectfont{}$\sigma$};
\path (13.0,9.6) node[text=black,anchor=base east] {\fontsize{10.8}{12.96}\selectfont{}$\tilde\sigma$};
\path (9.5,10.6) node[text=red,anchor=base east] {\fontsize{10.8}{12.96}\selectfont{}$\eta_i$};
\path (12.9,13.6) node[text=red,anchor=base east] {\fontsize{10.8}{12.96}\selectfont{}$u$};

% objects at depth 40:
\draw[fill=blue] (8,13.95) circle (0.08cm);
\draw[very thick,arrows=-to,blue] (8,16) -- (8,14.2);
\path (8,13.3) node[text=blue,anchor=base] {\fontsize{10.8}{12.96}\selectfont{}$(x,1)$};

\end{tikzpicture}%

\qquad
\begin{tikzpicture}[x=.3cm,y=-.3cm]

% objects at depth 60:
\path[fill=white!75!black,arrows=-to] (10.25333,13) -- (8,16) -- (4,13);
\draw[very thick,blue] (1,13) -- (11.5,13);

% objects at depth 50:
\draw[black] (8,16) -- (0,11);
\draw[dashed,black] (8,16) -- (14,8);
\draw[black] (8,16) -- (14,10);
\draw[dashed,black] (8,16) -- (0,10);
\draw[very thick,arrows=-to,red] (8,16) -- (4,13);
\draw[black] (4,13) -- (10.25333,12.99333);
\draw[very thick,blue] (6.6,12.6) -- (7.4,12.6);
\draw[very thick,arrows=-to,blue] (7,12.6) -- (7,10.6);
\path (6.8,12.0) node[text=blue,anchor=base east] {\fontsize{10.8}{12.96}\selectfont{}$(x,1)$};
\path (7.5,9.13333) node[text=red,anchor=base east] {\fontsize{10.8}{12.96}\selectfont{}$u$};
\path (8.76667,10.66667) node[text=black,anchor=base west] {\fontsize{10.8}{12.96}\selectfont{}$\tilde\sigma^\dual$};
\path (12.03333,10.53333) node[text=black,anchor=base west] {\fontsize{10.8}{12.96}\selectfont{}$\sigma^\dual$};
\path (4.3,14.2) node[text=red,anchor=base east] {\fontsize{10.8}{12.96}\selectfont{}$\eta_i$};

% objects at depth 40:
\draw[very thick,arrows=-to,red] (8,16) -- (8,8.1);
\path (8,14.54) node[anchor=base east] {\fontsize{10.8}{12.96}\selectfont{}$R$};
\path (11.6,13.5) node[text=blue,anchor=base west] {\fontsize{10.8}{12.96}\selectfont{}$\{\langle \cdot, (x,1) \rangle = s\}$};

\end{tikzpicture}%

\end{center}

Let us bound the height of an interior lattice point of
$\tilde\sigma^\dual$. Assume there is some $y \in \intr
\tilde\sigma^\dual \cap (\Z^{n+1})^*$ such that $\pro{y}{(x,1)} <
1$. Because $x \in P$ is a convex combination of vertices there is
some vertex $w \in P \times \{1\}$ such that $\pro{y}{w} <
1$. However, $y \in \intr \tilde\sigma^\dual \subseteq \intr
\sigma^\dual$ implies $0 < \pro{y}{w}$. This contradicts $\pro{y}{w}
\in \Z$. Now, $\pro{\,\cdot\,}{(x,1)} \le s$ is a valid inequality for
$R$, and by the above $\intr(kR) \cap (\Z^{n+1})^* = \emptyset$ for $k
\leq s^{-1}=\qd(P)$.

On the other hand, $u$ is a lattice point in $\intr
\tilde\sigma^\dual$ with $\pro{u}{(x,1)} = 1$. So $u \in \intr(kR)
\cap (\Z^{n+1})^*$ for $k > \qd(P)$. Hence, $r:=\cd(R) = \lfloor
\qd(P) \rfloor + 1$.

From Lemma \ref{sums} applied to $R$ and $(u,r) \in \pos(R \times
\{1\}) \cap (\Z^{n+2})^*$ we conclude that 
\[(u,r) = k (0,1) + (\eta_1, 1) + \cdots + (\eta_g,1) + (p,j)\]
for a natural number $k$, for (not necessarily different) non-zero
vertices $\eta_1, \ldots, \eta_g$ of $R$ and for a lattice point $p
\in (j R) \cap (\Z^{n+1})^*$ with the property that $(p,j)=(0,0)$ or
$1 \leq j \leq n+2-r$.
% and $p$ is not a multiple of some vertex $x_1, \ldots, x_g$. %

From $u \not\in (r-2)R$ and $(u,r-2)= (k-2) (0,1) +
(\eta_1, 1) + \cdots + (\eta_g,1) + (p,j)$ we conclude that $k-2 < 0$,
that is, $k \in \{0,1\}$. Further, if $k=1$, then $u \in (r-1)R
\setminus \intr((r-1)R)$ so that $1=\pro{u}{(x,1)}=(r-1)s$, that is,
$\qd(P) \in \Z$.

Let us first consider the case $k=0$. Since $u \in \intr(r R)$, we observe that $(u,r) \not\in \pos((\eta_1,1), \ldots, (\eta_g,1))$, thus, 
$(p,j) \not= (0,0)$. 
Therefore, $r=g+j$, and $u$ splits into a sum of at least $g +1 \geq r+1-(n+2-r) = 2 \lfloor \qd(P) \rfloor - n + 1$ 
non-zero lattice vectors in $\tilde\sigma^\dual$. Hence, Lemma~\ref{lemma:Batyrev-Nill}(3) yields that 
$P$ is a Cayley polytope of lattice polytopes in $\R^m$ with $m \leq n+1-(g+1) \leq 2 (n - \lfloor \qd(P) \rfloor)$.

It remains to deal with the case $k=1$. Here, we have already observed that $\qd(P) \in \Z$. 
If $(p,j)=(0,0)$, then $u$ splits into a sum of at least $g+1=r$ non-zero lattice points in $\tilde\sigma^\dual$, 
so Lemma~\ref{lemma:Batyrev-Nill}(\ref{item:lemma:Batyrev-Nill3}) yields that $P$ is the Cayley polytope of 
lattice polytopes in $\R^m$ with $m \leq n+1-(g+1) \leq n+1-\qd(P)$. Finally, if $(p,j) \not= (0,0)$, then 
$r = g+1+j$, so we again deduce from Lemma~\ref{lemma:Batyrev-Nill}(\ref{item:lemma:Batyrev-Nill3}) that $P$ is the Cayley polytope of 
$g+1=r-j \geq r-(n+2-r)=2r-n-2$ lattice polytopes in an ambient space of dimension $n+1-(2r-n-2) = 2(n-\qd(P)) +1$.
\end{proof}
\begin{remark}
  Statement and proof of Theorem~\ref{new} generalize 
  Theorem~3.1 in \cite{HNP09}, which proves 
  Conjecture~\ref{C2} in the case of {\em Gorenstein polytopes}. 
  A Gorenstein polytope $P$ with codegree $c$ can be characterized 
  by the property that $P$ is a $\Q$-normal lattice polytope with
  $\ad{(c P)}{1}$ being a lattice point.
\end{remark}
\begin{corollary}
  Let $P$ be an $n$-dimensional lattice polytope. If $n$ is odd and
  $\qd(P) > \frac{n+1}{2}$, or if $n$ is even and $\qd(P) \geq
  \frac{n+2}{2}$, then $P$ is a Cayley polytope.
\label{cor}
\end{corollary}
There is no obvious analogue  for {\em rational} polytopes. For instance, 
for $\epsilon > 0$, the $\Q$-codegree of $(1+\epsilon) \Delta_n$ equals $(n+1)/(1+\epsilon)$, so it gets arbitrarily close to $n+1$, 
however its lattice width is always strictly larger than one.

Theorem~\ref{new} proves Conjecture~\ref{C2}, if $\up{\qd(P)} = \cd(P)$. 
Therefore, using Proposition~\ref{estimate} we get the following new result.
\begin{corollary}
Conjecture~\ref{C2} holds, if $\NF(P)$ is Gorenstein and $P$ is $\Q$-normal.
\end{corollary}

If $P$ is smooth with $\cd(P) > \frac{n+2}{2}$, 
then it was shown in \cite{DDP09,DN10} that $P \cong P_0 * \cdots *
P_t$, where $t+1 = \cd(P)=\qd(P)$, and $P_0, \ldots, P_t$ have the
same normal fan. The proof relies on algebraic geometry, no purely
combinatorial proof is known.

\subsection{A sharper conjecture}
\label{q-conj}
We conjecture that in Corollary~\ref{cor} the condition $\qd(P) >
\frac{n+1}{2}$ should also be sufficient in even dimension. 
This is motivated by an open question in algebraic geometry, see Remark~\ref{ag-motivation}. 
We can prove this conjecture  in the case of lattice simplices.

\begin{proposition}
Let $P \subseteq \R^n$ be an $n$-dimensional rational simplex. Let $a_i$ be the
lattice distance of the $i$-th vertex of $P$ from the facet of $P$ not containing the vertex. Then
\[\nf(P) = \qd(P) = \sum_{i=0}^n \frac{1}{a_i}.\]
\label{simplex}
\end{proposition}
\begin{proof}
  Let $x$ be the unique point that has the same lattice distance $s$ from
  each facet. Then $\nf(P)^{-1} = \qd(P)^{-1} = s$.  Fix a basis
  $\{e_0,\ldots,e_n\}$ for $\R^{n+1}$ and consider the affine
  isomorphism $P \to \conv(a_0 e_0, \ldots, a_n e_n)=\{y \in
  \R_{\ge0}^{n+1} \,:\, \sum_{i=0}^n \nicefrac{y_i}{a_i} = 1\}\subset
  \R^{n+1}$ given by $y \mapsto (d_{F_0}(y), \ldots, d_{F_n}(y))$. The
  point $x$ is mapped to $c:=(s, \ldots, s)$, so
  $\nicefrac1s=\sum_{i=0}^n \nicefrac1{a_i}$.
\end{proof}

\begin{corollary}
  Let $P \subseteq \R^n$ be an $n$-dimensional lattice simplex. 
\begin{enumerate}
\item\label{item:cor:pyro1} If $\qd(P) >
  \frac{n+1}{2}$ (or $\qd(P) = \frac{n+1}{2}$ and $a_i \not= 2$ for
  some $i$), then $P$ is a lattice pyramid.
\item\label{item:cor:pyro2} If $\qd(P) \geq \frac{n+1}{2}$ and $P \not\cong 2 \Delta_n$, then $P$ has lattice width one.
\end{enumerate}
\label{pyro}
\end{corollary}
\begin{proof} 
  Assume that $P$ is not a lattice pyramid. Then $a_i \geq 2$ for all   $i=0, \ldots, n$. Hence,
  \begin{align*} 
    \qd(P) &= \sum_{i=0}^n \frac{1}{a_i} \leq \frac{n+1}{2}.
\end{align*} 
This proves \ref{item:cor:pyro1}. For \ref{item:cor:pyro2}, let us assume that $a_i=2$ for all $i=0, \ldots, n$. 
We consider the injective affine map $\R^n \to \R^n$, $y \mapsto (d_{F_1(y)}, \ldots, d_{F_n}(y))$. 
Note that the image of $P$ is $2 \Delta_n = \conv(0, 2e_1, \ldots, 2 e_n)$. 
Let us denote the image of $\Z^n$ by $\Lambda$. It satisfies $2 \Z^n \subseteq \Lambda \subseteq \Z^n$. 
If $\Lambda = \Z^n$, then $P \cong 2 \Delta_n$. Hence, 
our assumption yields that the reduction mod $2$ is a proper linear subspace $\Lambda/2\Z^n \subset (\Z/2\Z)^n$. 
Therefore,  it must satisfy an equation 
$\sum_{i \in I} x_i \equiv 0 \mod 2$ for some subset $\varnothing\ne I \subseteq \{1, \ldots, n\}$. The linear functional 
$\nicefrac12 (\sum_{i \in I} x_i)$ defines an element $\lambda \in\Lambda^*$ such that $\lambda(2 e_i)=1$ if
$i\in I$ and $0$ otherwise. Hence, $P$ has lattice width one in the 
direction of the pullback of $\lambda$. 
\end{proof}

\begin{example}
It is tempting to guess that $\qd(P) = \frac{n+1}{2}$ and $a_i = 2$
for all $i$ implies that $P \cong 2 \Delta_n$.
However, here is another example:
$ \conv \left[
  \renewcommand{\arraycolsep}{1.5pt}
  \renewcommand{\arraystretch}{.7}
  \begin{array}{>{\scriptstyle}r>{\scriptstyle}r>{\scriptstyle}r>{\scriptstyle}r>{\scriptstyle}r>{\scriptstyle}r}
    0&0&1&1\\
    0&1&0&1\\
    0&1&1&0
  \end{array}
\right] .$
\end{example}
A corresponding result for the codegree was proven in \cite{Nil08} where it is shown that a
lattice $n$-simplex is a lattice pyramid, if $\cd(P) \geq \frac{3}{4}
(n + 1)$. Let us stress that Conjecture~\ref{C2} is still open for
lattice simplices.

\section{Adjunction theory of toric varieties}\label{AG}

In this section, we explain the connection between  the previous combinatorial results and the adjunction theory of toric varieties.

\subsection{General notation and definitions.}
Let $X$ be a normal projective algebraic variety of dimension $n$ with
canonical class $K_X$ defined over the complex numbers. We assume
throughout that $X$ is {\em $\Q$-Gorenstein} of index $r$, i.e., $r$ is
the minimal  $r \in \N_{> 0}$ such that $r K_X$ is a Cartier divisor.
$X$ is called {\em Gorenstein}, if $r=1$.

Let $L$ be an ample line bundle (we will often use the same symbol for
the associated Weil divisor) on $X$. We use the additive notation to
denote the tensor operation in the Picard group, ${\rm Pic}(X).$ When
we consider  (associated) $\Q$-divisors the same additive notation
will be used for the operation in the group ${\rm Div(X)}\otimes\Q.$

Recall that $L$ is {\em nef}, resp. {\em ample}, if it has
non-negative, resp. positive, intersection with all irreducible curves
in $X.$ Moreover $L$ is said to be {\em big} if the global sections of
some multiple define a birational map to a projective space. If a line
bundle is nef, then being big is equivalent to having positive
degree. It follows that every ample line bundle is nef and big.
The pair $(X,L),$ where $L$ is an ample line bundle on $X$ is often
called a {\em polarized algebraic variety}.
The linear systems $|K_X+sL|$ are called {\em adjoint linear
  systems}. These systems define classical invariants which have been
essential tools in the existent classification of projective
varieties. In what follows we summarize what is essential to
understand the results in this paper. More details can be found in
1.5.4. and 7.1.1. of \cite{BS95}.

\begin{definition}
\item Let $(X,L)$ be a polarized variety. 
\begin{enumerate}
\item The {\em unnormalized spectral value} of $L$ is defined as
\[ \begin{array}{ccc}
 \qd(L) &:= &\sup \{s \in \Q \,:\, h^0(N(K_X + s L)) = 0   \text{ for all positive integers } N \\
 &&\text{  such that }N(K_X + s L)\text{  is an integral Cartier divisor} \}.\end{array}\]

Note that, $\qd(L) < \infty$ follows from $L$ being big.
\item The {\em nef value} of $L$ is defined as 
\[\nf(L) := \min \{s \in \R \,:\, K_X + s L \textup{ is nef}\}.\]
\end{enumerate}
 \end{definition}
 It was proven by Kawamata that $\nf(L)\in\Q$. 
Moreover if $r\nf=\frac{u}{v},$ where $u$ and $v$ are coprime,  then the  linear system $|m(vrK_X+uL)|$
is globally generated for a big enough integer $m.$ The corresponding morphism, $f:X\to
\P^{M}=\P(H^0(m(vrK_X+uL))),$  has  a Remmert-Stein factorization as $f=p\circ
\phi_{\nf},$ where $\phi_{\nf}:X\to Y$ is a morphism with connected
fibers onto a normal variety $Y,$ called the {\em nef value morphism}.
The rationality of $\qd(L)$ was only shown very recently
\cite[1.1.7]{BCHMcK10} as a consequence of the existence of the minimal model
program. 

Observe that the invariants above can be visualized as follows, see Figure~\ref{nef-cone}. 
Traveling from  $L$ in the direction of the vector $K_X$ in the Neron-Severi space ${\rm NS}(X)\otimes\R$ of 
divisors, $L+\frac{1}{\qd(L)}K_X$ is the meeting point with the cone of effective divisors ${\rm Eff}(X)$ and $L+\frac{1}{\nf(L)} K_X$ is the meeting point with the cone of nef-divisors ${\rm Nef}(X).$
We now summarize some well-known results which will be used in this section.
\begin{figure}
\centering
\def\svgwidth{10cm}
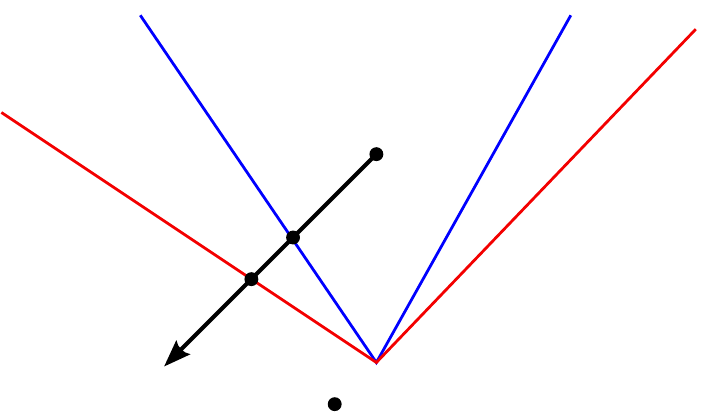
\caption{Illustrating $\qd(L)$ and $\nf(L)$}
\label{nef-cone}
\end{figure}
\begin{proposition}\label{kawa}In the above notation:
\begin{enumerate}
\item\label{item:prop:kawa1} $\nf(L)$ is the largest $s \in \Q$ such that $K_X + s L$ is nef but not ample.
\item\label{item:prop:kawa2} $\qd(L) \leq \nf(L), $ 
with equality if and only if $\phi_{\nf}$ is not birational.
\item\label{item:prop:kawa3} Let $r \nf(L) = u/v$ with coprime positive integers $u,v$. Then 
\[u \leq r (n+1),\]
in particular, $\nf(L) \leq r (n+1)$.
\item\label{item:prop:kawa4} $\qd(L) \leq n+1$.
\end{enumerate}
\end{proposition}
\begin{proof}
(\ref{item:prop:kawa1}) is proven in \cite[1.5.5]{BS95}. For (\ref{item:prop:kawa2}) observe that the  interior of the closure of the effective cone is the big cone, $\overline{{\rm Eff}(X)}^{\rm int}={\rm Big}(X).$ Recall that if a divisor is not big, then the map associated to the global sections has a lower-dimensional image.
It follows that the map is birational only when $\nf$ and $\nu$ do not coincide. A proof can be also found in \cite[7.1.6]{BS95}. 
(\ref{item:prop:kawa3})  is part of Kawamata's rationality theorem and  (\ref{item:prop:kawa4}) is proven in \cite[7.1.3]{BS95}.\end{proof}

\begin{remark}
  There are at least three other notions which are related to the
  unnormalized spectral value. The (non-negative) {\em spectral value}
  $\sigma(L) := n+1-\qd(L)$ was defined by Sommese in \cite{Som86}
  (compare this notion with the {\em degree} of lattice polytopes, \cite{BN07}). 
  Fujita defined in \cite{Fuj92} the (non-positive) {\em Kodaira
    energy} $\kappa\epsilon(L)$ as $-\qd(L)$, see also \cite{BT98}.
  Furthermore, the reciprocal $\qd(L)^{-1}$ is called the {\em effective
    threshold}, see e.g.  \cite{BCHMcK10}.
\end{remark}

There are several classifications of polarized varieties with large nef value. Fujita,  \cite{Fuj87}, proved that:
\begin{theorem}[Fujita 87~\cite{Fuj87}]
Let $(X,L)$ be a polarized normal Gorenstein variety  with $\dim(X)=n.$ Then
\begin{enumerate}
\item $\nf(L)\leq n$ unless $(X,L)=(\P^n, {\mathcal O}_{\P^n}(1)).$
\item  $\nf(L)< n $ unless 
\begin{enumerate}
\item $(X,L)$ as in $(1)$
\item $X$ is a quadric hypersurface and $L={\mathcal O}_X(1).$
\item $(X,L)=(\P^2, {\mathcal O}_{\P^n}(2)).$
\item $(X,L)=(\P(E),{\mathcal O}(1)),$ where $E$ is a vector bundle of rank $n$ over a nonsingular curve.
\end{enumerate}

\end{enumerate}
\end{theorem}
In the same paper Fujita also classifies the cases $\nf(L)\geq n-2$
and $\nf(L)\geq n-3.$ We will discuss this classification in the toric
setting and the induced classification of lattice polytopes with no
interior lattice points in the appendix.

\subsection{Toric geometry} We refer the reader who is unfamiliar with toric geometry to \cite{Ful93}. In what follows we will assume that $X$ is a {\em $\Q$-Gorenstein} toric variety of Gorenstein index $r$ and dimension $n.$ Let $L$ be an (equivariant) line bundle on $X.$
Let $N \cong \Z^n$, $\Sigma\subset N\otimes\R$ be the defining fan and denote by $\Sigma(i)$ the set of cones of $\Sigma$ of dimension $i.$ For $\nf\in\Sigma(i),$ $V(\nf)$ will denote the
associated invariant subvariety codimension $i.$

Recall that $L$ is nef (resp. ample) if and only if $L\cdot V(\rho_j)\geq 0$ (resp. $>0$) for all $\rho_j\in\Sigma(n-1),$ see for example \cite[3.1]{Mu02}.

There is a one-to-one correspondence between $n$-dimensional  toric
varieties polarized by an ample line bundle $L$ and $n$-dimensional
convex lattice polytopes
$P_{(X,L)}\subset M\otimes\R$ (up to translations by a lattice
vector), where $M$ is the lattice dual to $N.$ Under this
correspondence $k$-dimensional invariant subvarieties of $X$  are
associated with $k$-dimensional faces of $P_{(X,L)}$. More precisely,
if
\begin{equation}\label{polytope} P = \{x \in \R^n \,:\, A x \geq b\}\end{equation}
for an $m \times n$ integer matrix $A$ with primitive rows, and $b=(b_1,\ldots,b_m) \in \Z^m$ then $L=\sum (-b_i) D_i,$ where $D_i=V(\beta_i)$, 
for $\beta_i\in\Sigma(1)$, are the invariant divisors, generating the Picard group.

More generally, a nef line bundle ${\mathcal L}$
on a toric variety $X'$ defines a polytope $P_{\mathcal L}\subset \R^n,$ not necessarily of maximal dimension, whose integer points 
correspond to characters on the torus and form a basis of $H^0(X',\mathcal L)$. The edges of the polytope $P_{\mathcal L}$ correspond to the invariant curves whose intersection with ${\mathcal L}$ is positive.
In particular, the normal fan of $P_{\mathcal L}$ does  not necessarily coincide with the fan of $X'.$ It is the fan of a toric variety $X$ obtained by possibly contracting invariant curves on $X'.$ The contracted curves correspond to the invariant curves having zero intersection with ${\mathcal L}.$ Let $\pi:X'\to X$ be the contraction morphism. There is an ample line bundle $L$ on $X$ such that $\pi^*(L)={\mathcal L}.$
Because the dimension of the polytope equals  the dimension of the image of the map defined by the global sections one sees immediately that $P_{\mathcal L}$ has maximal dimension if and only if 
${\mathcal L}$ is big.

% To summarize, a lattice polytope might have a twofold interpretation
% of its data. There is  a unique polarized toric variety $(X,L)$
% associated to it. At the same time one could associate to it a variety
% $X'$ birational to $X$  and a nef line bundle $\mathcal L$ on $X'.$
% The variety $X$ is obtained by contracting  certain invariant curves
% of $X'$, and the line bundle $\mathcal L$ has intersection zero
% exactly with these curves.
% By abuse of notation one often considers the pairs  $(X',{\mathcal
%   L})$ and $(X,L)$ as defined by the same combinatorial data, but it
% is important to note that only one of them is a polarized toric
% variety.

\subsection{Adjoint bundles (compare with Section~1)}

Let $(X,L)$ be the polarized variety defined by the polytope
(\ref{polytope}). Observe that for any $s\in\Q_{>0}$ the polytope
$\adP{s} := \{x \in \R^n \,:\, A x \geq b+s \eins\},$ with $\eins =
(1, \ldots, 1)^\transpose,$ corresponds to the $\Q$-line bundle
$sK_X+L.$ With this interpretation it is clear that
\[ \qd(P)=\qd(L)\text{ and } \nf(P)=\nf(L) \]

\begin{remark} %{\rm 
Proposition \ref{prop:M(P)} gives us a geometric interpretation of
these invariants. Let  $k\in\Z$ such that $kM(P)$ is a  lattice
polytope and let $Y$ be the associated toric variety. The polytope $P$
is a facet of $M(P)$ and thus the variety $X$ is an invariant divisor
of $Y.$ Moreover, the projection $M(P) \twoheadrightarrow P$ induces a
rational surjective map $Y\to \P^1$ whose generic fiber (in fact all
fibers but the one at $\infty$) are isomorphic to $X$.
\end{remark}

\begin{remark}
From an inductive viewpoint, it would be  desirable to know how ``bad'' the 
singularities of $\adP{1}$ can get, if we start out with a ``nice'' polytope $P$. 
However, this seems to be very hard. Traditionally, there is another
way, the so called ``onion-skinning'' of a polytope, see \cite{HaaseSchicho,Oga07}
via the {\em interior polytope} $\oni{P}:= \conv(\intr(P) \cap \Z^n)$. 
Recall that the lattice points of $\adP{1}$ correspond to the global
sections of $K_X+L_P.$ If the line bundle $K_X+L_P$ is globally
generated (equivalently nef)  then $\adP{1}=\oni{P},$ but in general they
might be different. 
Obviously, $\oni{P} \subseteq \adP{1}$, with equality if and only if
$\adP{1}$ is a lattice polytope. In \cite{Oga07} Ogata examined the
case of smooth polytopes of dimension at most three with interior
lattice points. He proves the following:
\begin{itemize}
\item[-]{\em in dimension two}, $\adP{1}$
equals $\oni{P}$, and it is even a smooth polytope, see
\cite[Lemma~5]{Oga07}.  

\item[-]{\em In dimension three},  \cite[Prop.~3]{Oga07}, it is claimed  that by
successively forgetting facet inequalities  
(corresponding to blow-downs) it is possible to obtain a smooth polytope $P' \supseteq P$ with 
$\ad{P'}{1} = \ad{P}{1} = \oni{P}$ and $\nf(P') \leq 1$. Moreover, while $\oni{P}$ may not be smooth anymore, 
Proposition~4 of \cite{Oga07} says that singular points of cones over
$(\P^2, O(2))$ and $(\P^1 \times \P^1, O(1,1))$ are  
the only possible singularities occurring at the toric fix points of $X_{\oni{P}}$. 
\end{itemize}

It would be desirable to understand what happens in higher dimensions,
for instance we expect the answer to the following question to be negative:

{\em Let $P$ be a smooth four-dimensional polytope with interior
  lattice points. Is $\ad{P}{1}$ still a lattice polytope?}
\end{remark}

\subsection{Admissible polarized toric varieties (compare with Section~2)}

In the language above,  Proposition \ref{prop:projection} states that
if $(X,L)$ is  a polarized $\Q$-Gorenstein toric variety then there is a
finite sequence of maps of  toric varieties
\[X_k\rightarrow X_{k-1}\rightarrow\ldots\rightarrow
X_{2}\rightarrow X_1\rightarrow X_0=X\] polarized by ample
line bundles $L_i.$  In fact by considering the polytope
$P=P_{(X,L)}$ Proposition \ref{prop:projection} gives a projection $P
\twoheadrightarrow Q$ from the linear space ${\rm Aff}(
\adP{\frac{1}{\qd(L)}}).$ The projection defines a map of fans:
$\Sigma_Q\to \Sigma_P$ and in turn a map of toric varieties:
$X_1\rightarrow X.$
Notice that $\dim(X_1)=\dim(X)-\dim(\adP{\frac{1}{\qd(L)}}).$ Let
$L_1$ be the polarization defined by $Q$ on $X_1$. Starting again with
$(X_1,L_1)$ 
we look at the corresponding projection $Q\twoheadrightarrow Q_1$ and
so on. Notice that the sequence will stop when $\mu(X_{k-1})=\mu(X_k)$
and $\core(Q_k)$ is a single (rational) point.
We remark that the $\Q$-codegree has been defined for any polytope while
the spectral value is defined only for $\Q$-Gorenstein varieties.
In more generality the singularities are quite subtle and it is not at
all clear how to proceed within algebraic geometry.
For this purpose we will call a a polarized $\Q$-Gorenstein toric
variety {\em admissible}, if in the sequence above $X_i$ is
$\Q$-Gorenstein for every $0\leq i\leq k.$ Recall that the lattice
points of $N\core(Q_k)$ correspond to the global sections
$H^0(N(K_{X_k}+\qd(L_k)L_K)),$ for  an integer $N$ such that
$N(K_{X_k}+\qd(L_k)L_K)$ is an integral line bundle. Then
Proposition \ref{prop:projection} reads as follows:

\begin{proposition}
Let $(X,L)$ be  an admissible polarized  $\Q$-Gorenstein toric
variety. There is a finite sequence of maps of  toric varieties
\[X_k \rightarrow X_{k-1} \rightarrow\ldots\rightarrow
X_{2} \rightarrow X_1 \rightarrow X_0=X\] polarized by ample
line bundles $L_i$ such that $\qd(L_i)\geq\qd(L_{i-1})$ for $1\le i\le
k$ and $H^0(N(K_{X_k}+\qd(L_k)L_K))$ consists of a single section for
an integer $N$ such that $N(K_{X_k}+\qd(L_k)L_K)$ is an integral line
bundle.
\end{proposition}

\begin{example}
  The polytope in Figure \ref{figEx} defines an admissible polarized
  $\Q$-Go\-ren\-stein toric variety. Let $(X,L)$ be the associated
  polarized toric variety. The (unnormalized) spectral value satisfies 
  $\qd(L)=\qd(P)=\frac{3}{4}.$ The polytope has the following
  description:
  \begin{align*}
    P&=\left(\; \left(\begin{array}{c}x\\y\\z\end{array}\right)\;
      \left|\begin{array}{rcl}
          x&\geq& 0\\ y&\geq& 0\\z&\geq& 0\\x+y&\leq& 4\\ hx+2z&\leq& 2h\end{array}\right.\right) &&\text{ if }h\text{ odd and  }\\[.1cm]
    P&=\left(\left(\begin{array}{c}x\\y\\z\end{array}\right)\;
      \left|\begin{array}{rcl} x&\geq& 0\\ y&\geq& 0\\z&\geq&
          0\\x+y&\leq& 4\\ kx+z&\leq&
          2k\end{array}\right.\right)&&\text{ if }h=2k \text{ for some
      integer $k$}
  \end{align*}
For simplicity let us assume that $h$ is odd.
From the polytope one sees that ${\rm Pic}(X)$ is generated by $D_1,\ldots,D_5$ with the following linear relations:
\[D_1\sim hD_5+D_4,\; D_2\sim D_4,\; D_3\sim 2D_5\]
Moreover  $L=4D_4+2hD_5$ and $K_X=-3D_4-(h+3)D_5$ giving $4K_X+3L=(2h-12)D_5,$ which is effective for $h\geq 6.$
The first projection onto $Q$ defines in this case an invariant subvariety $X_1$ which is isomorphic to $\P^2$ blown up at one point. Moreover  $L_1=L|_{X_1}=4l-2E,$ where
$l$ is the pull back of the hyperplane line bundle on $\P^2$ and $E$ is the exceptional divisor. The variety $X_1$ is smooth and therefore $\Q$-Gorenstein of index $1.$ Starting again with $(X_1,L_1)$
we have $\nu(L_1)=1$ and $X_2\cong\P^1$ with $L_2={\mathcal O}_{\P^1}(2)$ which give $\qd(L_2)=1$ and $H^0(K_{X_2}+L_2)=H^0(\mathcal O_{X_2}).$
\end{example}

It would be desirable to have criteria for a toric polarized $\Q$-Gorenstein variety to be admissible.

\subsection{The main result (compare with Section~3)}\label{fiber} As explained in
\cite{HNP09} and in \cite{DDP09} the toric variety  $X,$ defined by a
Cayley polytope,  
$$P=P_0 * \cdots * P_t$$ 
has a prescribed  birational morphism to the toric projectivized bundle
$X=\P(H_0\oplus H_1\oplus\cdots\oplus H_t)$ over a toric variety $Y.$
The variety $Y$ is  defined by a common refinement of the inner normal
fans of the polytopes $P_i$. Moreover, the polytopes $P_i$ are
associated to the nef line bundles $H_i$ over $Y.$
As a consequence of Theorem~\ref{main} we get the following result.

\begin{proposition} 
Let $(X,L)$ be a polarized $\Q$-Gorenstein toric variety.  Suppose $q \in \Q_{>0}$ such that $2q\le n$ and 
no multiple of $K_X + (n+1-q)L$ which is Cartier has non-zero global sections.  Then there is a proper birational toric morphism $\pi: X' \rightarrow X$, where $X'$ is the projectivization of a sum of line bundles on a toric variety of dimension at most $\lfloor 2q \rfloor$ and $\pi^*L$ is isomorphic to $O(1)$.
\end{proposition}
\begin{proof}
The assumption $2q\le n$ implies that $\mu(L)\geq \frac{n+2}{2}.$ Theorem \ref{main} gives the conclusion.
\end{proof}

\begin{remark}\label{conj}
It is conjectured in Subsection~\ref{q-conj} that $\qd(L) >
\frac{n+1}{2}$ should suffice in Corollary~\ref{cor}. One
algebro-geometric statement which hints at this possibility is a
conjecture by Beltrametti and Sommese, \cite[7.1.8]{BS95}, that states
that $\qd(L) > \frac{n+1}{2}$ should imply $\qd(L)=\nf(L)$, when the
variety is nonsingular. 
% Notice that in the smooth case $\qd(L) =\cd(P) \in \Z$ by
% Proposition~\ref{estimate}, so $\qd(L) \geq \frac{n+2}{2}$. 
Moreover, it was also conjectured in \cite{FS89} 
that if $\qd(L) > 1$, then $\qd(L) = p/q$ for integers $0 < q \leq p
\leq n+1$. In particular, $\qd(L) > \frac{n+1}{2}$ would again imply
$\qd(L) \in \Z$.
%\marginnote{Future work: It may be possible to find counterexamples
%to $Q$-normality in the case of a singular dual defective polytope by
%taking the Cayley sum of four generic lattice polygons (i.e., $\cd(P)
%= 4 > 3 = (n+1)/2$).}
\label{ag-motivation}
\end{remark}

Let $A$ be the set of lattice points of a lattice polytope $P$, and
let $X_A$ be the (not necessarily normal) toric 
variety embedded in $\P^{|A|-1}$. Then there is an irreducible
polynomial, called the {\em $A$-discriminant}, which is of 
degree zero if and only if the dual variety $X^*_A$ 
is not a hypersurface (i.e., $X_A$ has {\em dual defect}), see \cite{GKZ94}.

\begin{proposition}\label{dual}
Let $P$ be a lattice polytope with $\qd(P) \geq \frac{3n+4}{4}$, such
that $\qd(P) \not\in \N$, respectively, $\qd(P) \geq \frac{3n+3}{4}$,
if $\qd(P) \in \N$. Then $X_A$ has dual defect.
\label{def-prop}
\end{proposition}

\begin{proof}
By Theorem~\ref{new}, $P$ is a Cayley polytope of at least $n+1-d$
lattice polytopes in $\R^d$, where the assumptions yield that $n+1-d
\geq d+2$. Then Proposition~6.1 and Lemma~6.3 in \cite{DFS07} imply
the desired result. Note that in the notation of \cite{DFS07}
$m=n+1-d$, $r=d$, and $c = m-r \geq2$.
\end{proof}

For smooth polarized toric varieties is was verified that
the assumption $\qd(L) > \frac{n+2}{2}$ is equivalent to the variety
having dual defect, see \cite{DN10}. Moreover, smooth dual defective
varieties are necessarily $\Q$-normal ($\qd(L)=\nf(L)$) by
\cite{BFS92}. By the results of \cite{DiR06,DDP09} this implies that
the associated lattice polytope is a smooth Cayley polytope of
$\qd(L)=\cd(P)$ many smooth lattice polytopes with the same normal
fan.
On the other hand, it has recently been shown \cite{CC07,Est10} that
all lattice points in a (possibly singular) dual defective polytope
have to lie on two parallel hyperplanes. 
However, it is not true that all Cayley polytopes, or polytopes of
lattice width $1$, are dual defective, even in the nonsingular case.
%But with certain restrictions on the dimension and numbers of Cayley
%addends it is true (at least in the smooth case)
Therefore, the main question is whether the following strengthening of
Proposition~\ref{def-prop} may be true, see \cite{DN10}:
\begin{question} 
  Is $(X,L)$ dual defective, if $\qd(L) > \frac{n+2}{2}$\,?
\end{question} 

\appendix\section{Fujita's classification results}\label{sec:appendix}

\renewcommand{\thetheorem}{A.\arabic{theorem}}

In this section  we provide a translation of the results in \cite[Theorem 2 and 3']{Fuj87}. A straightforward corollary gives the classification of smooth polytopes of dimension three with no interior lattice points. 
One could derive a more extensive classification from all the results contained in  \cite[Theorem 2 and 3']{Fuj87} and from later work 
such as \cite{BDiT03,Nak97}. This would require a more elaborate explanation which goes beyond the scope of this paper.

\begin{theorem}[Fujita 87~\cite{Fuj87}] Let $P$ be an $n$-dimensional lattice polytope such that its normal fan is Gorenstein. Then
\begin{enumerate}
\item If $\nf(P)>n$, then $P\cong \Delta_n.$
\item If $n-1<\nf(P)\leq n$, then $P\cong 2\Delta_2$ or $P\cong P_0*P_1*\ldots*P_{n-1}$ where the
$P_i$ are parallel  intervals.
\item \label{blowy}  If  $P$ is  smooth and  $n-2<\nf(P)\leq n-1$, then $P$ is one of the following polytopes:
\begin{enumerate}
\item 
There is a smooth $n$-dimensional polytope $P'$ and a unimodular simplex $S \not\subseteq P$ such that $$P'=P\cup S$$ and
$P\cap S$ is a common facet of $P$ and $S$.

\item $\adP{\frac{1}{n-1}}$ is a point.
\item $P=2\Delta_3, 3\Delta_3, 2\Delta_4$.
\item There is a projection $\pi: P \twoheadrightarrow \Delta_1\times\Delta_1$
\item  There is a projection $\pi: P \twoheadrightarrow  2\Delta_2$ and 
the polytopes $\pi^{-1}(m_i)$ have the same normal fan, where $m_i$ are the vertices of $2\Delta_2.$
\item $P\cong P_0*P_1*\ldots*P_{n-2},$ where the $P_i$ are smooth polygons with the same normal fan.
\end{enumerate}
\end{enumerate}
\end{theorem}

Note that in (\ref{blowy})(a)  $P$ is given by a vertex truncation of $P'$ (compare with Figure~\ref{vertex-skeleton}), 
corresponding to a blow-up at a smooth point. 
The following result is a simple corollary of the previous classification. It was also obtained in a slightly weaker form by
Ogata \cite[Proposition~1]{Oga07},  using combinatorial methods.

\begin{corollary}
Let $P$ be a smooth $3$-dimensional polytope with no interior lattice points. Then $P$ is of one of the following types.
\begin{enumerate}
\item $P=\Delta_3, 2\Delta_3, 3 \Delta_3.$ 
\item There is a projection $P \twoheadrightarrow \Delta_2$, where any preimage of each vertex is an interval. Equivalently there are  $a,b,c\in\Z$ such that
$$P = \conv \left[
  \begin{smallmatrix}
    0&0&0&a&b&c \\
    0&1&0&1&0&0\\
    0&0&1&0&0&1&
  \end{smallmatrix}\right]$$ 
\item There is a projection $P \twoheadrightarrow 2 \Delta_2$, where any preimage of each vertex is an interval. Equivalently there are  $a,b,c\in\Z$ such that
$$P = \conv \left[
  \begin{smallmatrix}
    0&0&0&a&b&c \\
    0&2&0&2&0&0\\
    0&0&2&0&0&2&
  \end{smallmatrix}\right]$$ 

\item There is a projection $P \twoheadrightarrow  \Delta_1\times\Delta_1.$ Equivalently there are  $a,b,c\in\Z$ such that
$$P = \conv \left[
  \begin{smallmatrix}
    0&0&1&1&0&1&1&0 \\
    0&1&0&1&1&1&0&0\\
    0&0&0&0&a&b&c&a+b-c
  \end{smallmatrix}\right]$$

\item $P=P_0 *P_1$, where $P_0$ and $P_1$ are smooth polygons with the same normal fan.
\item There is a smooth $3$-dimensional polytope $P'$ with no interior lattice points and a unimodular simplex $S \not\subseteq P$ 
such that $$P'=P\cup S$$ and
$P\cap S$ is a common facet of $P$ and $S$.
\end{enumerate}

\end{corollary}

\renewcommand{\bibname}[1]{#1, }

\end{document}